\documentclass[onefignum,onetabnum]{siamonline250211}

\usepackage{lipsum}
\usepackage{amsfonts}
\usepackage{graphicx}
\usepackage{epstopdf}
\usepackage{algorithmic}
\usepackage{amsmath}
\allowdisplaybreaks \allowdisplaybreaks[4]
\ifpdf
  \DeclareGraphicsExtensions{.eps,.pdf,.png,.jpg}
\else
  \DeclareGraphicsExtensions{.eps}
\fi

\usepackage{enumitem}
\setlist[enumerate]{leftmargin=.5in}
\setlist[itemize]{leftmargin=.5in}

\newcommand{\rd}{\mathrm d}

\newcommand{\hE}{\mathbb E}
\newcommand{\hN}{\mathbb N}
\newcommand{\hP}{\mathbb P}
\newcommand{\hR}{\mathbb R}

\newcommand{\cC}{\mathcal C}
\newcommand{\cF}{\mathcal F}

\newcommand{\cR}{\mathcal R}
\newcommand{\hF}{\mathbb F}

\newcommand{\<}{\langle}
\renewcommand{\>}{\rangle}

\newsiamremark{remark}{Remark}
\newsiamremark{hypothesis}{Hypothesis}
\crefname{hypothesis}{Hypothesis}{Hypotheses}
\newsiamthm{claim}{Claim}
\newsiamremark{fact}{Fact}
\crefname{fact}{Fact}{Facts}

\newsiamthm{assumption}{Assumption}
\newsiamthm{example}{Example}

\headers{Fundamental weak convergence theorem for SVIEs}{X.~Dai, Q.~Yin, and D.~Jin}

\title{Fundamental weak convergence theorem for stochastic Volterra integral equations and its applications\thanks{Submitted to the editors DATE.
\funding{This work is supported by National Natural Science Foundation of China (Nos.\ 12401547, 12201228, 12471391), and Yunnan Fundamental Research Project (No.\ 202501AU070074), and Scientific Research and Innovation Project of Postgraduate Students in the Academic Degree of Yunnan University (No.\ KC-252513126).}}}

\author{Xinjie Dai\thanks{School of Mathematics and Statistics, Yunnan University, Kunming 650500, Yunnan, China (\email{dxj@ynu.edu.cn}).}
\and Qijiao Yin\thanks{School of Mathematics and Statistics, Yunnan University, Kunming 650500, Yunnan, China (\email{yinqijiao@stu.ynu.edu.cn}).}
\and Diancong Jin\thanks{School of Mathematics and Statistics, Huazhong University of Science and Technology, Wuhan 430074, China; 
Hubei Key Laboratory of Engineering Modeling and Scientific Computing, Huazhong University of Science and Technology, Wuhan 430074, China (Corresponding author, \email{jindc@hust.edu.cn}).}}

\usepackage{amsopn}

\ifpdf
\hypersetup{
  pdftitle={Fundamental weak convergence theorem for stochastic Volterra integral equations and its applications},
  pdfauthor={X.~Dai, Q.~Yin, and D.~Jin}
}
\fi

\begin{document}

\maketitle

\begin{abstract}
We study weak convergence rates of numerical approximations for stochastic Volterra integral equations (SVIEs), a class of non-Markovian models that arises naturally in stochastic volatility modeling and other fields. The intrinsic non-Markovian nature prevents the direct application of classical weak error techniques developed for finite-dimensional Markov processes. To overcome this difficulty, we combine a Markovian lifting technique with a domino argument, Taylor expansions, and Fr\'echet differential calculus for path-dependent functionals, and establish a fundamental weak convergence theorem for nonsingular SVIEs, providing a unified approach to the weak error analysis for a broad class of numerical approximations. As applications, we derive the first-order weak convergence rate for the stochastic theta method and the Wong--Zakai approximation. Our results relax existing assumptions for Euler-type schemes by removing the boundedness requirement on the diffusion coefficient. Furthermore, to the best of our knowledge, this work provides the first weak convergence result for Wong--Zakai approximations of SVIEs. Numerical experiments for a stochastic volatility model corroborate the theoretical convergence rate. 
\end{abstract}



\section{Introduction}

Stochastic Volterra integral equations (SVIEs) extend stochastic differential equations (SDEs) to a non-Markovian setting by incorporating memory effects and capturing more general forms of temporal dependence. For example, in mathematical finance, classical SDE-based models often fail to capture the persistent memory effects and rough behaviors observed in financial time series. SVIEs provide a flexible framework for incorporating such non-Markovian features. In particular, rough volatility models, formulated through Volterra-type dynamics for the volatility process, could reproduce important stylized facts of financial markets, including the pronounced short-maturity implied volatility smile \cite{AbiJaber2019, BayerBook, ElEuchRosenbaum2019}. Beyond finance, SVIEs have found extensive applications in modeling anomalous diffusion in physics and biology, as well as viscoelastic materials \cite{Didier2022, LiLiuLu2017}. In this work, we consider the following SVIE:\ 
\begin{align} \label{eq.SVIE} 
X_t = X_0 + \int_0^t \mu(t, s, X_s) \rd s + \int_0^t \sigma(t, s, X_s) \rd W_s, \qquad t \in[0, T]. 
\end{align}
Here, $X_0 \in \hR^d$, $\mu: [0, T] \times [0, T] \times \hR^d \rightarrow \hR^d$, $\sigma: [0, T] \times [0, T] \times \hR^d \rightarrow \hR^{d \times m}$, and $W$ is an $m$-dimensional Brownian motion defined on some complete filtered probability space $(\Omega,\cF,\hP,\hF)$ with the filtration $\hF=\{\cF_t\}_{0 \leq t \leq T}$ satisfying the usual conditions. A prominent class of SVIEs frequently encountered in the literature takes the form 
\begin{align}\label{SVEwithK} 
X_t = X_0 + \int_0^t K_1(t,s) \tilde{\mu}(s, X_s) \rd s + \int_0^t K_2(t,s) \tilde{\sigma}(s, X_s) \rd W_s, 
\end{align}
where the kernel $K_i$ ($i = 1,2$) is said to be regular if $\lim _{s \uparrow t} |K_i(t,s)| < \infty$, and singular if $\lim _{s \uparrow t} |K_i(t,s)| = \infty$ (cf.\ \cite{Hamaguchi2025}).

In financial applications such as option pricing and risk management, computing the expectation of a path-dependent or terminal payoff function of the asset profile (i.e., $\mathbb{E}[f(X_T)]$) is an important objective. Since closed-form solutions to SVIEs are rarely accessible, the development and rigorous mathematical justification of efficient numerical approximations are of paramount theoretical and practical importance. Over the past decade, considerable progress has been made in the strong error analysis of numerical approximations for SVIEs; see e.g., \cite{Alfonsi2024, JourdainPages2025, LiHuangHu2022, RichardTanYang2021, Zhang2008JDE}. Moreover, the optimality of the strong rates of the Euler--Maruyama and Milstein methods is validated by the study of asymptotic error distributions \cite{FukasawaUgai2023, LiuHuGao2025, NualartSaikia2023}. Although a bound on the weak error can be obtained from bounds on the strong error, it is suboptimal in general. Thus, weak error estimates are much more desirable when approximating expectations of functionals. Unfortunately, weak error analysis for SVIEs is notoriously challenging and remains largely underdeveloped compared with that for SDEs, as the future increment of the solution depend heavily on its past trajectory, thereby inducing a non-Markovian structure.

For specific classes of SVIEs, such as rough volatility models, there have been some recent advances in the weak error analysis of their numerical discretizations. For instance, in the singular kernel case, building on the framework proposed by Friz et al.\ \cite{FrizSalkeld2025AAP} and employing Gaussian computation and Malliavin calculus, Alfonsi and Kebaier \cite{Alfonsi2026weak} established the optimal weak error estimate of the Euler-type method for a generalized rough volatility model. We also refer to \cite{BayerFukasawa2022, BayerHall2022, BonesiniJacquier2023, Gassiat2023} for related weak error analysis of rough volatility models, where the volatility process is represented as a functional of Riemann--Liouville fractional Brownian motion. However, for general SVIEs of the form \eqref{eq.SVIE} or \eqref{SVEwithK}, we are only aware of the work by \cite{BrasFukasawa2025} and \cite{FengZhang2023} regarding the weak convergence analysis of numerical methods. For the SVIE \eqref{SVEwithK} with regular kernels, Bras and Fukasawa \cite{BrasFukasawa2025} recently proposed an infinite-dimensional analytical framework together with a generalized It\^o formula to establish the first-order weak convergence rate for two classes of Euler-type methods, with and without kernel discretization. Based on stochastic Taylor expansions and a functional It\^o formula, Feng and Zhang \cite{FengZhang2023} analyzed the weak convergence rate for the cubature method applied to the SVIE \eqref{SVEwithK} with regular kernels. Despite these advances, a unified framework for analyzing weak errors---independent of specific schemes arguments---remains absent for general SVIEs. Specifically, a \emph{fundamental weak convergence theorem} for numerical SVIE schemes, analogous to the classical weak approximation theory for SDEs (see, e.g., \cite{MilsteinTretyakov2004, WangZhaoZhang2024}), has yet to be established. The present work is devoted to filling this gap.

To address the difficulties caused by the non-Markovian nature of SVIEs, an effective strategy is the Markovian lifting, which embeds the path-dependent solutions into an infinite-dimensional state space of continuous trajectories and further recover a desired Markovian-type structure \cite{AbiJaber2019Multifactor, FriesenGerhold2026, Hamaguchi2024}. In this spirit and following \cite{BrasFukasawa2025}, we introduce the two-parameter operator $P_{r,t}:\Phi_{2T}\to\Phi_{2T}$ (see also \eqref{eq.def:P_rt}), where $\Phi_{2T}$ denotes the space of $\mathbb{R}^d$-valued continuous trajectories on $[0,\infty)$ with support in $[0,2T]$. The operator $P_{r,t}$ exhibits a Markovian-type structure along the $\Phi_{2T}$-valued process $\{Y_t\}$: $P_{r, t}(Y_{t})_u = Y_{t+r}(u)$ for $u \geq 0$, where $Y_t(u) := P_{t,0}(\widetilde{\mathbf{0}})_u$ is ingeniously linked to the solution process of \eqref{eq.SVIE} through the relation $X_t=X_0+Y_t(0)$. This Markovian-type structure and its numerical counterpart enable a systematic decomposition of the global weak error via a discrete-time domino (or telescoping) argument, effectively reducing the global weak convergence analysis to the evaluation of a one-step weak error between $P_{r,t}$ and its numerical analogue, which serves as the core ingredient in establishing the global weak convergence rate. In \cite{BrasFukasawa2025}, the one-step weak error of Euler-type methods for \eqref{SVEwithK} was analyzed by applying a generalized It\^o-type formula. In this work, we introduce a different approach to control the one-step weak error by leveraging classical It\^o calculus combined with Fr\'echet differential calculus and Taylor expansions for path-dependent functionals. On the basis of it, we formulate a fundamental weak convergence theorem (i.e., Theorem \ref{Thm.ImportantII}) applicable to a broad class of approximation methods, thereby providing a unified methodology.

The effectiveness and versatility of our fundamental weak convergence theorem are highlighted through two applications:\ 
\begin{enumerate}
\item \textbf{The Stochastic Theta Method}:\ By applying our fundamental weak convergence theorem, we prove that the stochastic theta method achieves a first-order weak convergence rate. In addition, we remove the restriction on the boundedness of $\sigma$ and relax the fifth-order continuous differentiability of $\mu$ and $\sigma$ with respect to the state variable in the existing literature for Euler-type methods (see Remark \ref{remark.compare}).

\item \textbf{The Wong--Zakai Approximation}:\ We utilize the proposed fundamental weak convergence theorem to derive the weak error estimate for Wong--Zakai approximations. To the best of our knowledge, this provides the first weak convergence result for Wong--Zakai approximations of SVIEs.
\end{enumerate}

The rest of the paper is organized as follows. In Section \ref{sec.FWCT}, we introduce the basic formulation and assumptions for nonsingular SVIEs and establish a fundamental weak convergence theorem for their numerical approximations. In Sections \ref{sec.STM} and \ref{sec.WZA}, we apply the proposed weak convergence theorem to derive weak error estimates for the stochastic theta method and the Wong--Zakai approximation, respectively. In Section \ref{sec.NumericalExperiments}, numerical experiments for a stochastic volatility model are performed to corroborate the theoretical results. Finally, some auxiliary lemmas are collected in Appendix \ref{appendix_AuxiLemm}.

\section{Preliminaries and fundamental weak convergence theorem}
\label{sec.FWCT}

In this section, we introduce some preliminaries and present the fundamental weak convergence theorem for numerical approximations of SVIEs.

Throughout this paper, unless otherwise specified, we use the following notations. Let $\hE $ denote the expectation under the probability measure $\hP$. The symbols $\lfloor \,\cdot\, \rfloor$ and $\lceil \,\cdot\, \rceil$ represent the floor and ceiling functions, respectively. Let $C$ represent a generic constant, which may take different values at each occurrence but is always independent of the step size $h$. Fix $T > 0$, and let $\Phi_{2T}$ denote the space of $\mathbb{R}^d$-valued continuous trajectories on $\hR_+ := [0,\infty)$ with support in $[0,2T]$, endowed with the supremum norm $\|\cdot\|_{\infty}$. For continuously differentiable $\varphi \in \Phi_{2T}$, we denote by $\dot{\varphi}$ its derivative. If $\varphi$ is Lipschitz continuous, we write $[\varphi]_{Lip}$ for its Lipschitz constant. We denote by $\widetilde{\mathbf{0}}$ the path on $\hR_+$ identically equal to $0 \in \hR^d$. For operators $Q_1, \cdots, Q_J$, we denote their composition by $\prod_{j=1}^J Q_j := Q_1 \circ \cdots \circ Q_J$.

Let $(E, \| \cdot \|)$ be a real Banach space and $\ell \in \hN_+$. We denote by $\mathcal{C}_b^{\ell}(E; \mathbb{R})$ the space of all functions $f: E \to \mathbb{R}$ that are $\ell$ times continuously Fr\'echet differentiable and possess bounded derivatives up to order $\ell$ on $E$. For $1 \le k \le \ell$, the $k$-th Fr\'echet derivative of $f$ at $x \in E$ is denoted by $\nabla^k f(x) \in \mathcal{L}^k(E; \mathbb{R}),$ where $\mathcal{L}^k(E; \mathbb{R})$ represents the space of $k$-linear mappings from $E^k$ into $\mathbb{R}$. We note that each $\nabla^k f(x)$ is symmetric and continuous. We define the semi-norm on $\mathcal{C}_b^{\ell}(E; \mathbb{R})$ as
\begin{align*}
\|f\|_{\mathcal{C}_b^{\ell}(E)} := \sum_{k=1}^{\ell} \|\nabla^kf\|_\infty, \qquad \mbox{where}~~~ 
\|\nabla^kf\|_\infty:= \sup_{x \in E}\, \sup_{\|h_1\|\le 1,\ldots,\|h_k\|\le 1}|\<\nabla^k f(x),\otimes_{i=1}^k h_i\rangle|.
\end{align*}

\subsection{Stochastic Volterra integral equations}

Under the following Assumption \ref{assumption} with $\ell = 1$, the SVIE \eqref{eq.SVIE} admits a unique strong solution $X \in \cC([0,T]; L^{p} (\Omega; \hR^d))$ with $p \geq 2$. See \cite{AbiJaber2019, JourdainPages2025} for more details on the well-posedness of SVIEs.

\begin{assumption} [Regularity of coefficients] 
\label{assumption} 
$\mu \in \cC_b^{0,0,\ell}([0, T] \times [0, T] \times \mathbb{R}^d; \hR^{d})$ and $\sigma \in \cC_b^{0,0,\ell}([0, T] \times [0, T] \times \mathbb{R}^d; \hR^{d \times m})$ for some $\ell \in \hN_+$, where the subscript $b$ means that all derivatives of $\mu$ and $\sigma$ with respect to the third variable, up to order $\ell$, are bounded. 
\end{assumption}

For convenience, we extend the functions $\mu$ and $\sigma$ to the domain $\hR_+ \times \hR_+ \times \hR^d$. Specifically, we set
\begin{align*} 
\mu(t,s,x) = 0, \quad \sigma(t,s,x) = 0 \quad \text{for } (t,s) \notin [0,2T] \times [0,2T] \text{ and } x \in \hR^d.
\end{align*}
This extension allows us to write the SVIE \eqref{eq.SVIE} in a unified form for all $(t,s) \in \hR_+^2$, which simplifies the subsequent analysis and the definition of the associated `semigroup' type operators. Moreover, the extended functions $\mu$ and $\sigma$ retain the same regularity assumptions as in Assumption \ref{assumption}.

Inspired by the approach in \cite{BrasFukasawa2025}, we introduce a `semigroup' type operator associated with the SVIE \eqref{eq.SVIE} in order to handle its non-Markovian nature. For $t\in[0,T]$ and $r\in[0,T-t]$, define the operator $P_{r,t}: \Phi_{2T} \to \Phi_{2T}$ by $\varphi\mapsto P_{r,t}(\varphi)$ with 
\begin{align} \label{eq.def:P_rt}
P_{r,t}(\varphi)_u
= \varphi_{r+u} + \int_t^{t+r} \mu(t+r+u,s,\widetilde{X}_s^{t,\varphi})\,\rd s + \int_t^{t+r} \sigma(t+r+u,s,\widetilde{X}_s^{t,\varphi})\,\rd W_s,
\end{align}
for $u\ge0$. Here $\{\widetilde{X}_v^{t,\varphi}\}_{t\le v\le T}$ denotes the solution to the SVIE
\begin{align}\label{eq.def:tildeX}
\widetilde{X}_v^{t,\varphi}
= X_0 + \varphi_{v-t} + \int_t^v \mu(v,s,\widetilde{X}_s^{t,\varphi})\,\rd s + \int_t^v \sigma(v,s,\widetilde{X}_s^{t,\varphi})\,\rd W_s. 
\end{align}
By taking $t=0$ and $\varphi=\widetilde{\mathbf{0}}$, the uniqueness of the solution to \eqref{eq.SVIE} implies that $\widetilde{X}_v^{0,\widetilde{\mathbf{0}}}=X_v$ for all $v\in[0,T]$. In fact, the solution to \eqref{eq.SVIE} admits the representation
\begin{align}\label{eq.sol_XY}
X_t = X_0 + Y_t(0), \qquad t\in[0,T], 
\end{align}
where $Y_t(u) := P_{t,0}(\widetilde{\mathbf{0}})_u$ for $u \geq 0$. In other words, 
\begin{align*}
Y_t(u) = \int_0^t \mu(t+u, s, X_s) \rd s + \int_0^t \sigma(t+u, s, X_s) \rd W_s, 
\qquad t \in[0, T], \quad u \geq 0. 
\end{align*}
In light of \eqref{eq.def:P_rt} and \eqref{eq.def:tildeX}, it holds that 
\begin{align*}
P_{r, t}(Y_{t})_u 
&= Y_{t}(r+u) + \int_{t}^{t+r} \mu(t+r+u, s, \widetilde{X}_s^{t,Y_{t}}) \rd s + \int_{t}^{t+r} \sigma(t+r+u, s, \widetilde{X}_s^{t,Y_{t}}) \rd W_s\\
&= \int_0^{t} \mu(t+r+u, s, X_s) \rd s + \int_0^{t} \sigma(t+r+u, s, X_s) \rd W_s\\
&\quad + \int_{t}^{t+r} \mu(t+r+u, s, \widetilde{X}_s^{t,Y_{t}}) \rd s + \int_{t}^{t+r} \sigma(t+r+u, s, \widetilde{X}_s^{t,Y_{t}}) \rd W_s, 
\end{align*}
which together with the fact that $\widetilde{X}_s^{t,Y_{t}} = X_s$ for $s \in[t,t+r]$ yields the semigroup-type property:\ $P_{r, t}(Y_{t}) = Y_{t+r}$, or more explicitly,
\begin{align} \label{eq.semigroup:X}
P_{r, t}(Y_{t})_u = Y_{t+r}(u), \qquad u \geq 0.
\end{align}

\subsection{Fundamental weak convergence theorem}
\label{sub:FWCT}

Let $\{t_j = jh\}_{j=0}^N$ be a uniform partition of $[0,T]$ with step size $h = T/N$, where $N \in \mathbb{N}_+$. For $j \in \{0,1,\cdots,N-1\}$, we consider a numerical approximation of \eqref{eq.def:P_rt} with $r =h$ and $t = t_j$. The resulting approximation of $P_{h,t_j}$ defines a numerical `semigroup' type operator $\bar{P}_{h,t_j} : \Phi_{2T} \to \Phi_{2T}$ by $\varphi \mapsto \bar{P}_{h,t_j}(\varphi)$. For $0\le i\le N-1$ and $0\le j\le N-i$, we denote 
\begin{align*} 
\bar{P}_{t_j, t_i} = \prod_{k=1}^{j}\bar{P}_{h, t_{i+j-k}}.
\end{align*}
By convention, we set $\bar{P}_{0,t_i} = \mathrm{Id}$ for $i \in \{0,1,\cdots,N\}$. In analogy with \eqref{eq.sol_XY}, the numerical approximation of $X_{t_k}$ is defined by
\begin{align} \label{eq:Xtk}
\bar{X}_{t_k} = X_0 + \bar{Y}_{t_k}(0), \qquad k = 0,1,\ldots,N,
\end{align}
where $\bar{Y}_{t_k}(u) := \bar{P}_{t_k,0}(\widetilde{\mathbf{0}})_u$ for $u \ge 0$. Then, the semigroup-type property follows from 
\begin{small}
\begin{align} 
\label{eq.semigroup:tildeX} 
\bar{P}_{h, t_j} (\bar{Y}_{t_j}) = \bar{P}_{h, t_j} ( \bar{P}_{t_j,0}(\widetilde{\mathbf{0}}) ) = \bar{P}_{h, t_j} \left( \prod_{k=1}^{j}\bar{P}_{h, t_{j-k}}(\widetilde{\mathbf{0}}) \right) = \prod_{k=1}^{j+1}\bar{P}_{h, t_{j+1-k}}(\widetilde{\mathbf{0}}) = \bar{P}_{t_{j+1},0}(\widetilde{\mathbf{0}}) = \bar{Y}_{t_{j+1}}, 
\end{align}
\end{small}
for any $j \in \{0, 1, \ldots, N-1\}$.

\begin{example}[Euler--Maruyama method] 
For $j \in \{0,1,\cdots,N-1\}$, applying the Euler--Maruyama method to \eqref{eq.def:P_rt} with $r =h$ and $t = t_j$ defines a numerical `semigroup' type operator $\bar{P}_{h, t_j}: \Phi_{2T} \to \Phi_{2T}$, $\varphi \mapsto \bar{P}_{h, t_j}(\varphi)$, where
\begin{align*}
\bar{P}_{h, t_j}(\varphi)_u 
&= \varphi_{h+u} + \mu(t_{j+1}+u, t_j, \widetilde{X}_{t_j}^{t_j,\varphi})\, h + \sigma(t_{j+1}+u, t_j, \widetilde{X}_{t_j}^{t_j,\varphi})\, (W_{t_{j+1}}-W_{t_j}) \\
&= \varphi_{h+u} + \mu(t_{j+1}+u, t_j, X_0 + \varphi_0)\, h + \sigma(t_{j+1}+u, t_j, X_0 + \varphi_0)\, (W_{t_{j+1}}-W_{t_j}), \notag
\end{align*}
for $u \ge 0$. 
Then, recalling the definition of $\bar{Y}$, \eqref{eq:Xtk}, and \eqref{eq.semigroup:tildeX} indicates that for any $0\le i,\,k\le N$,
\begin{align*}
\bar{Y}_{t_k}(t_i) 
&= \bar{P}_{h,t_{k-1}}(\bar{Y}_{t_{k-1}} )_{t_i} \\
&= \bar{Y}_{t_{k-1}}(t_{i+1})+ \mu( t_{k+i}, t_{k-1}, X_0 + \bar{Y}_{t_{k-1}}(0))h \\
&\quad + \sigma( t_{k+i}, t_{k-1}, X_0 + \bar{Y}_{t_{k-1}}(0) ) (W_{t_k}-W_{t_{k-1}})\\
&=\bar{Y}_{t_{k-1}}(t_{i+1})+ \mu( t_{k+i}, t_{k-1}, \bar{X}_{t_{k-1}})h + \sigma( t_{k+i}, t_{k-1}, \bar{X}_{t_{k-1}}) (W_{t_k}-W_{t_{k-1}}).
\end{align*}
Thus, iterating the above recurrence relation with the fact that $\bar{Y}_0 = \widetilde{\mathbf{0}}$ yields 
\begin{align*}
\bar{Y}_{t_k}(t_i)
&= \sum_{j=1}^{k} \mu( t_{k+i}, t_{j-1}, \bar{X}_{t_{j-1}})h + \sum_{j=1}^{k} \sigma( t_{k+i}, t_{j-1}, \bar{X}_{t_{j-1}}) (W_{t_j}-W_{t_{j-1}}). 
\end{align*}
Finally, by taking $i=0$ in the above formula and using \eqref{eq:Xtk}, one can recover the Euler--Maruyama method for SVIE \eqref{eq.SVIE} as follows:\ for $k \in \{0, 1, \ldots, N\}$, 
\begin{align*}
\bar{X}_{t_k}
= X_0 + \sum_{j=1}^{k} \mu( t_{k}, t_{j-1}, \bar{X}_{t_{j-1}})h + \sum_{j=1}^{k} \sigma( t_{k}, t_{j-1}, \bar{X}_{t_{j-1}}) (W_{t_j}-W_{t_{j-1}}). 
\end{align*}
\end{example}

To state the fundamental weak convergence theorem for numerical approximations of SVIE \eqref{eq.SVIE}, we impose the following two assumptions.

\begin{assumption} [Regularity of approximations] 
\label{assum.barP is C^1}
For all $k \in \{0, 1, \ldots, N-1\}$, it holds that
\begin{align*}
\bar P_{t_k, 0}(\widetilde{\mathbf{0}}) \in \cC^1 ([0, \infty); \hR^d) 
\qquad \mbox{and} \qquad 
\mathbb{E} \left( \big[ \bar P_{t_k, 0} (\widetilde{\mathbf{0}}) \big]_{Lip}^{\nu} \right) \leq C \qquad 
\mbox{for some~~~} \nu > 0. 
\end{align*}
\end{assumption}

For $g: \Phi_{2T} \rightarrow \hR$, we define the following operators 
\begin{align} \label{eq.redef of P}
\boldsymbol{P}_{r, t} \, g(\varphi ) = \mathbb{E} \big[ g(P_{r, t}(\varphi )) \big] 
\qquad \mbox{and} \qquad 
\bar{\boldsymbol{P}}_{t_j, t_i} \, g(\varphi )= \mathbb{E} \big[ g(\bar{P}_{t_j, t_i}(\varphi )) \big], 
\end{align}
where $\varphi \in \Phi_{2T}$, $t \in [0, T]$, $r \in [0, T-t]$, $0\le i\le N-1$, and $0\le j\le N-i$.

\begin{assumption} [One-step approximation error] 
\label{assum.importantI}
Let $\ell$ and $\nu$ be as specified in Assumptions \ref{assumption} and \ref{assum.barP is C^1}, respectively. For any $\varphi \in \Phi_{2T} \cap \mathcal{C}^1([0, \infty); \mathbb{R}^d)$ and any $g \in \mathcal{C}_b^{\ell}( \Phi_{2T}; \mathbb{R})$, there exists a constant $C > 0$ independent of $\varphi$ and $h$, such that 
\begin{align*}
| (\bar{\boldsymbol{P}}_{h, t_k} - \boldsymbol{P}_{h, t_k}) g(\varphi) | 
\leq C \big( 1 + [\varphi]_{Lip}^{\nu} \big) h^{q+1},\qquad \forall\, k \in \{0, 1, \dots, N-1\}
\end{align*}
holds for some $q > 0$. Here, the constant $C$ depends on $g$ only through its semi-norm $\|g\|_{\mathcal{C}_b^{\ell}( \Phi_{2T})}$.
\end{assumption}

\begin{proposition} 
\label{Prop.ImportantI}
Let $X$ and $\bar{X}$ denote the exact solution and its numerical approximation defined by \eqref{eq.SVIE} and \eqref{eq:Xtk}, respectively. Suppose that Assumptions \ref{assumption}, \ref{assum.barP is C^1}, and \ref{assum.importantI} are satisfied. Then, for any given test function $f \in \mathcal{C}_{b}^{\ell} (\mathbb{R}^d; \mathbb{R})$, there exists a constant $C > 0$ independent of $h$ such that 
\begin{align*}
\big| \mathbb{E}[f(\bar{X}_T)] - \mathbb{E}[f(X_T)] \big| \leq C h^{q},
\end{align*}
where the rate $q > 0$ of weak convergence is prescribed by Assumption \ref{assum.importantI}.
\end{proposition}

The above proposition shows that the global weak convergence order is one order lower than the local weak convergence order, which is consistent with the case of SDEs. To facilitate the proof of Proposition \ref{Prop.ImportantI}, we introduce the following auxiliary lemma. For $\varphi \in \Phi_{2T}$ and $f:\, \mathbb{R}^d \rightarrow \mathbb{R}$, define
\begin{align} 
\label{eq.def_gk}
g_{k}(\varphi ) 
= \mathbb{E} \big[ f(P_{t_{N-k}, t_{k}}(\varphi)_0 + X_0) \big], 
\qquad k \in \{ 0, 1, \dots, N-1 \}. 
\end{align}

\begin{lemma} \label{lem.def of g} 
Let Assumption \ref{assumption} hold with $\ell \geq 1$. If $f \in \cC_{b}^{\ell} (\mathbb{R}^d; \mathbb{R})$, then for each $k \in \{0,1, \dots, N - 1\}$, the functional $g_k$ belongs to $\cC_b^{\ell}( \Phi_{2T}; \hR)$. Moreover, there exists a constant $C > 0$, independent of $k$ and $h$, such that 
\begin{align*}
\|g_k\|_{\cC_b^{\ell}( \Phi_{2T})}\le C\|f\|_{\cC_{b}^{\ell} (\mathbb{R}^d)}. 
\end{align*} 
\end{lemma}

\begin{proof} 
For $k \in \{0,1, \dots, N - 1\}$ and $\varphi \in \Phi_{2T}$, we write $\widetilde X = \widetilde X^{t_k,\varphi}$ for short in this proof. Then it follows from \eqref{eq.def:P_rt} that 
\begin{align*}
P_{t_{N-k}, t_{k}}(\varphi)_0 + X_0 
= \varphi_{T-t_{k}} + \int_{t_{k}}^{T} \mu(T, s, \widetilde X_s) \rd s + \int_{t_k}^{T} \sigma(T, s, \widetilde X_s) \rd W_{s} + X_0. 
\end{align*}
We define the tangent process \(\{Z_{v}^{\varphi}\}_{v \in[t_k,T]}\) of $\widetilde X$ by 
\begin{align*}
\< Z_{v}^{\varphi}, \eta\>
&= \eta_{v-t_{k}} + \int_{t_{k}}^{v} \nabla_x \mu(v, s, \widetilde X_s) \< Z_s^{\varphi}, \eta\> \rd s + \int_{t_{k}}^{v} \nabla_x \sigma(v, s, \widetilde X_s)\< Z_s^{\varphi}, \eta\> \rd W_s,
\end{align*}
where $\eta \in \Phi_{2T}$ and $v \in [t_k,T]$. Then one can deduce that 
\begin{align} \label{eq:gketa}
\<\nabla g_k(\varphi), \eta\>
= \mathbb E \bigg[ \nabla f \Big( P_{t_{N-k}, t_{k}}(\varphi)_0 + X_0 \Big) \< Z_{T}^{\varphi}, \eta\> \bigg]. 
\end{align}
Let $\psi_v^\eta := \mathbb{E}[|\< Z_v^{\varphi},\eta\> |^{2}]$. Using the Cauchy--Schwarz inequality and It\^o's isometry, along with Assumption \ref{assumption}, yields 
\begin{align*}
\psi_v^\eta 
\leq 3\|\eta\|_\infty^2 + 3 \int_{t_k}^v (T \| \nabla_x \mu \|_\infty^2 + \|\nabla_x \sigma\|_{\infty}^2 ) \psi_s^\eta \rd s.
\end{align*}
An application of Gr\"onwall's inequality leads to
\begin{align}\label{eq:veta}
\psi_v^\eta
\leq 3 \|\eta\|_\infty^2 \exp \big( 3 (v-t_k) (T \| \nabla_x \mu \|_\infty^2 + \|\nabla_x \sigma\|_{\infty}^2 \big), \qquad \forall\, v \in [t_k,T]. 
\end{align}
As a consequence of the Cauchy--Schwarz inequality, \eqref{eq:gketa}, and \eqref{eq:veta}, one can arrive at 
\begin{align*}
|\<\nabla g_k(\varphi), \eta\>|^2
\leq \|\nabla f\|_\infty^2 \psi_T^\eta \leq C\|\nabla f\|_\infty^2 \|\eta\|_\infty^2. 
\end{align*}
Taking the supremum over $\eta \in \Phi_{2T}$ with $\|\eta\|_\infty\le 1$ leads to $\|\nabla g_k\|_{\infty}\le C\|\nabla f\|_\infty$, which completes the proof when $\ell=1$. The cases $\ell \ge 2$ follow from a similar argument and are omitted for the sake of brevity. 
\end{proof}

\begin{proof} [Proof of Proposition \ref{Prop.ImportantI}] 
For $f\in \cC_b^\ell (\mathbb{R}^d; \mathbb{R})$, define the associated functional $\widetilde{f}:\, \Phi_{2T} \rightarrow \mathbb{R}$ as $\widetilde{f}(\varphi) = f(\varphi_0 + X_0)$, where $\varphi_0$ and $X_0$ denote the initial values of $\varphi$ and $X$, respectively. Clearly, $\widetilde f \in \cC_b^\ell (\Phi_{2T}; \mathbb{R})$. Applying the identities $\widetilde{\mathbf{0}} = Y_0 = \bar{Y}_0$ and the semigroup-type properties \eqref{eq.semigroup:X} and \eqref{eq.semigroup:tildeX}, one obtains 
\begin{align*}
& X_T - X_0 
= Y_T ( 0 ) 
= P_{h, t_{N-1}}\circ \cdots \circ P_{h, t_0} (\widetilde{\mathbf{0}})_0 
= \left( \prod_{j=0}^{N-1} P_{h, t_{N-1-j}} (\widetilde{\mathbf{0}}) \right)_0, \\
& \bar{X}_T - X_0 
= \bar{Y}_T ( 0 ) 
= \bar{P}_{h, t_{N-1}}\circ \cdots \circ \bar{P}_{h, t_0} (\widetilde{\mathbf{0}})_0 
= \left( \prod_{j=0}^{N-1} \bar{P}_{h, t_{N-1-j}} (\widetilde{\mathbf{0}}) \right)_0, 
\end{align*} 
which together with the relation \eqref{eq.redef of P} shows 
\begin{small}
\begin{align*}
& \mathbb{E} f(X_T) 
= \mathbb{E} f \left( \left( \prod_{j=0}^{N-1}{P_{h, t_{N-1-j}} (\widetilde{\mathbf{0}})} \right)_0 + X_0 \right) 
= \mathbb{E} \widetilde{f} \left( \prod_{j=0}^{N-1}{P_{h, t_{N-1-j}} (\widetilde{\mathbf{0}})} \right) 
= \left( \prod_{j=0}^{N-1}{\boldsymbol{P}_{h, t_j}}\widetilde{f} \right) (\widetilde{\mathbf{0}}), \\
& \mathbb{E} f(\bar X_T) 
= \mathbb{E} f \left( \left( \prod_{j=0}^{N-1}{\bar P_{h, t_{N-1-j}} (\widetilde{\mathbf{0}})} \right)_0 + X_0 \right) 
= \mathbb{E} \widetilde{f} \left( \prod_{j=0}^{N-1}{\bar P_{h, t_{N-1-j}} (\widetilde{\mathbf{0}})} \right) 
= \left( \prod_{j=0}^{N-1}{\bar{\boldsymbol{P}}_{h, t_j}}\widetilde{f} \right) (\widetilde{\mathbf{0}}). 
\end{align*} 
\end{small}
Now the weak error can be reformulated as a telescoping sum as follows:\ 
\begin{align} \label{eq.weak error}
\mathbb{E} f(\bar{X}_T) - \mathbb{E} f(X_T) 
&= \left( \prod_{j=0}^{N-1}{\bar{\boldsymbol{P}}_{h, t_j}}\widetilde{f} \right) (\widetilde{\mathbf{0}}) - \left( \prod_{j=0}^{N-1}{\boldsymbol{P}_{h, t_j}}\widetilde{f} \right) (\widetilde{\mathbf{0}}) \notag \\ 
&= \sum_{k=0}^{N-1} \left( \bar{\boldsymbol{P}}_{t_k, 0}\circ ( \bar{\boldsymbol{P}}_{h, t_k}-\boldsymbol{P}_{h, t_k} ) \circ \prod_{j = k+1}^{N-1}{\boldsymbol{P}_{h, t_j}}\widetilde{f} \right) (\widetilde{\mathbf{0}}),
\end{align}
where the last step used the convention $\prod_{j=N}^{N-1}{\boldsymbol{P}_{h, t_j}}\widetilde{f} = \widetilde{f}$.

Recalling \eqref{eq.redef of P}, \eqref{eq.semigroup:X}, and \eqref{eq.def_gk} implies that $\prod_{j = k+1}^{N-1}{\boldsymbol{P}_{h, t_j}}\widetilde{f} (\widetilde{\mathbf{0}}) = g_{k+1}(\widetilde{\mathbf{0}})$ for all $k \in \{0,1, \dots, N - 1\}$. Then it follows from \eqref{eq.weak error} and \eqref{eq.redef of P} that 
\begin{align*}
\left| \mathbb{E} f(\bar{X}_T)-\mathbb{E} f(X_T) \right|
&\leq \sum_{k=0}^{N-1} \Big| \left( \bar{\boldsymbol{P}}_{t_k, 0} \circ \left( \bar{\boldsymbol{P}}_{h, t_k}-\boldsymbol{P}_{h, t_k} \right) g_{k+1} \right) (\widetilde{\mathbf{0}}) \Big| \nonumber \\
&= \sum_{k=0}^{N-1} \Big| \hE \Big[ 
\left( \bar{\boldsymbol{P}}_{h, t_k}-\boldsymbol{P}_{h, t_k} \right) g_{k+1} ( \bar P_{t_k, 0} (\widetilde{\mathbf{0}}) ) \Big] \Big|. 
\end{align*} 
In view of Lemma \ref{lem.def of g} and the fact $\widetilde f \in \cC_b^\ell (\Phi_{2T}; \mathbb{R})$, we have $g_{k} \in \cC_b^{\ell}( \Phi_{2T}; \hR)$ for all $k \in \{1, 2, \ldots, N\}$. Finally, by Assumptions \ref{assum.barP is C^1} and \ref{assum.importantI}, one can conclude that 
\begin{align*}
\left| \mathbb{E} f(\bar{X}_T)-\mathbb{E} f(X_T) \right|
\leq \sum_{k=0}^{N-1} C \left( 1 + \mathbb{E} \left( \big[ \bar P_{t_k, 0} (\widetilde{\mathbf{0}}) \big]_{Lip}^{\nu} \right) \right) h^{q+1} 
\leq C h^{q},
\end{align*}
which completes the proof. 
\end{proof}

Building upon Proposition \ref{Prop.ImportantI}, we now present the fundamental weak convergence theorem, which provides an alternative to Assumption \ref{assum.importantI} and offers more practical criteria for verification. Define the shift operator $\tau_h:\cC ([0, \infty); \hR^d) \to \cC ([0, \infty); \hR^d)$ by $( \tau_h \varphi )_u = \varphi_{h+u} $ for $u \geq 0$. For any $g \in \cC_b^{2q+2} ( \Phi_{2T}; \hR)$ and any $k \in \{ 0, 1, \dots, N-1 \}$, the functions $g(P_{h,t_k}(\varphi))$ and $g(\bar P_{h,t_k}(\varphi))$ admit Taylor expansions around $\tau_h \varphi $ as follows:\ 
\begin{align*}
& g(P_{h,t_k}(\varphi))-g( \tau_h \varphi ) = \sum_{m=1}^{2q+1} \frac{1}{m!} \left\< \nabla^m g( \tau_h \varphi ), (P_{h,t_k}(\varphi)- \tau_h \varphi )^{\otimes m} \right\> + \cR_{2q+2}, \\
& g(\bar P_{h,t_k}(\varphi))-g( \tau_h \varphi ) = \sum_{m=1}^{2q+1} \frac{1}{m!} \left\< \nabla^m g( \tau_h \varphi ), (\bar P_{h,t_k}(\varphi)- \tau_h \varphi )^{\otimes m} \right\> + \bar \cR_{2q+2},
\end{align*}
where the remainders are given by 
\begin{small}
\begin{align*}
&\cR_{2q+2} = \frac{1}{(2q+1)!} \int_0^1 (1-\xi)^{2q+1} \left\<\nabla^{2q+2} g( \tau_h \varphi + \xi (P_{h,t_k}(\varphi)- \tau_h \varphi )), (P_{h,t_k}(\varphi)- \tau_h \varphi )^{\otimes (2q+2)} \right\> \rd \xi,\\
&\bar \cR_{2q+2} = \frac{1}{(2q+1)!} \int_0^1 (1-\xi)^{2q+1} \left\< \nabla^{2q+2} g( \tau_h \varphi + \xi (\bar{P}_{h,t_k}(\varphi)- \tau_h \varphi ) ), (\bar{P}_{h,t_k}(\varphi)- \tau_h \varphi )^{\otimes (2q+2)} \right\> \rd \xi.
\end{align*}
\end{small}
Next, we impose the following assumption to guarantee that Assumption \ref{assum.importantI} holds with $\ell=2q+2$.

\begin{assumption} \label{assum.importantII}
Let $q > 0$ satisfy $2q+1 \in \hN_+$ and $g \in \cC_b^{2q+2} ( \Phi_{2T}; \hR)$. There exists a constant $C$ such that for any $k \in \{ 0, 1, \dots, N-1 \}$, $\varphi \in \Phi_{2T}\cap\cC^1 ([0, \infty); \hR^d)$, $m \in \{1, 2, \cdots, 2q+1\}$, and $\xi \in [0, 1]$, 
\begin{align*}
&\left| \hE \left\< \nabla^m g( \tau_h \varphi ), (P_{h,t_k}(\varphi)- \tau_h \varphi )^{\otimes m} - (\bar P_{h,t_k}(\varphi)- \tau_h \varphi )^{\otimes m} \right\> \right| 
\leq C(1 + [\varphi]_{Lip}^{2q+2})h^{q+1}, \\
&\left| \hE \left\<\nabla^{2q+2} g( \tau_h \varphi + \xi (P_{h,t_k}(\varphi)- \tau_h \varphi )), (P_{h,t_k}(\varphi)- \tau_h \varphi )^{\otimes (2q+2)} \right\> \right| 
\leq C(1 + [\varphi]_{Lip}^{2q+2})h^{q+1}, \\
&\left| \hE \left\<\nabla^{2q+2} g( \tau_h \varphi + \xi (\bar P_{h,t_k}(\varphi)- \tau_h \varphi )), (\bar P_{h,t_k}(\varphi)- \tau_h \varphi )^{\otimes (2q+2)} \right\> \right| 
\leq C(1 + [\varphi]_{Lip}^{2q+2})h^{q+1}, 
\end{align*}
where $C > 0$ depends on $g$ only through its semi-norm $\|g\|_{\cC_b^{2q+2}( \Phi_{2T})}$, but is independent of $\varphi$ and $h$. 
\end{assumption}

Based on the definition \eqref{eq.redef of P}, the triangle inequality, and Assumption \ref{assum.importantII}, we immediately obtain 
\begin{align*}
&\quad\ | (\bar{\boldsymbol{P}}_{h, t_k} - \boldsymbol{P}_{h, t_k}) g(\varphi) | =\left| \hE [g(P_{h,t_k}(\varphi))-g( \tau_h \varphi )] - \hE [g(\bar P_{h,t_k}(\varphi))-g( \tau_h \varphi )] \right|\\
&\leq\sum_{m=1}^{2q+1} \frac{1}{m!} \left | \hE \left[ \left\< \nabla^m g( \tau_h \varphi ), (P_{h,t_k}(\varphi)- \tau_h \varphi )^{\otimes m} - (\bar P_{h,t_k}(\varphi)- \tau_h \varphi )^{\otimes m} \right\> \right] \right | \\
&\quad + | \hE (\cR_{2q+2} ) | + | \hE ( \bar \cR_{2q+2}) |\\
&\le C(1 + [\varphi]_{Lip}^{2q+2})h^{q+1},
\end{align*}
which verifies Assumption \ref{assum.importantI} with $\ell = \nu = 2q+2$. Thus, we can state the following theorem.

\begin{theorem} [Fundamental weak convergence theorem] 
\label{Thm.ImportantII} 
Suppose that Assumptions \ref{assumption}, \ref{assum.barP is C^1}, and \ref{assum.importantII} hold with $\ell = \nu = 2q+2$ and $q > 0$. Then, for any given test function $f \in \mathcal{C}_{b}^{2q+2} (\mathbb{R}^d; \mathbb{R})$, there exists a constant $C > 0$ independent of $h$ such that 
\begin{align*}
\big| \mathbb{E}[f(\bar{X}_T)] - \mathbb{E}[f(X_T)] \big| \leq C h^{q}.
\end{align*} 
\end{theorem}

Compared with Proposition \ref{Prop.ImportantI}, Theorem \ref{Thm.ImportantII} exhibits more advantages in practical applications, as Assumption \ref{assum.importantII} is typically more readily verifiable than Assumption \ref{assum.importantI}. To illustrate this advantage, we apply Theorem \ref{Thm.ImportantII} in the subsequent two sections to establish weak error estimates for the stochastic theta method and the Wong--Zakai approximation.

\section{Stochastic theta method}
\label{sec.STM}
The stochastic theta method exhibits favorable numerical stability properties, whose strong error analysis has been investigated in \cite{Conte2018DCDSB, LiHuangHuWen2021}. 
 In this section, we further apply the fundamental weak convergence theorem (i.e., Theorem \ref{Thm.ImportantII}) to analyze the weak convergence rate of the stochastic theta method for SVIE \eqref{eq.SVIE}, defined by 
\begin{align}
\bar{X}_{t_n}
&= X_0 + \sum_{i=0}^{n-1} \int_{t_i}^{t_{i+1}} \theta \mu(t_{n}, s, \bar{X}_{t_{i+1}}) \rd s + \sum_{i=0}^{n-1} \int_{t_i}^{t_{i+1}} (1-\theta) \mu(t_{n}, s, \bar{X}_{t_i}) \rd s \notag \\
&\quad + \sum_{i=0}^{n-1} \int_{t_i}^{t_{i+1}} \sigma(t_{n}, t_i, \bar{X}_{t_i}) \rd W_s, 
\qquad n = 1,2,\cdots,N, \label{eq:SthetaSVIE} 
\end{align}
where $\theta \in [0,1]$ is a parameter. Here, for the second variable of the drift coefficient $\mu$, we do not discretize it, since the deterministic integrals can be evaluated exactly or efficiently approximated by suitable quadrature rules. In contrast, for the second variable of the diffusion coefficient $\sigma$, we discretize it to achieve high computational efficiency in simulating stochastic integrals. Throughout this section, we always assume that the step size $h \in (0,1)$ is sufficiently small such that the unique solvability condition $h \theta \| \nabla_x \mu\|_{\infty} < 1$ holds. For $t \in [0, T]$, there is a continuous-time version 
\begin{align}
&\bar{X}_t = X_0 + \int_0^t \theta \mu(t, s, \bar{X}_{\bar{s}}) + (1-\theta) \mu(t, s, \bar{X}_{\underline{s}}) \rd s + \int_0^t \sigma(t, \underline{s}, \bar{X}_{\underline{s}}) \rd W_s, \label{eq:SthetaSVIEcontinuous}
\end{align}
where $\underline{s} := \left\lfloor s/h \right\rfloor h$ and $\bar{s} := \left\lceil s/h \right\rceil h$ for $s \in [0,T]$.

Next, we demostrate that the stochastic theta method \eqref{eq:SthetaSVIE} fits into the framework developed in Subsection \ref{sub:FWCT}, which allows us to apply the fundamental weak convergence theorem directly. For $j \in \{0,1,\cdots,N-1\}$, we apply the stochastic theta method to \eqref{eq.def:P_rt} with $r =h$ and $t = t_j$ and define the resulting numerical `semigroup' type operator $\bar{P}_{h,t_j}: \Phi_{2T} \to \Phi_{2T}$ by $\varphi\mapsto \bar{P}_{h,t_j}(\varphi)$ with 
\begin{align} 
\bar{P}_{h, t_j}(\varphi)_u 
&= \varphi_{h+u}+ \int_{t_j}^{t_{j+1}} \theta \mu(t_{j+1}+u, s, \widehat X_{t_{j+1}}^{t_j,\varphi} ) + (1-\theta) \mu(t_{j+1}+u, s, \widehat X_{t_j}^{t_j,\varphi} ) \rd s \notag\\
&\quad + \int_{t_j}^{t_{j+1}} \sigma(t_{j+1}+u, t_j, \widehat X_{t_j}^{t_j,\varphi} ) \rd W_s, \qquad u \geq 0. \label{eq.barP_rt}
\end{align}
Here, $\widehat X_{t_j}^{t_j,\varphi} = X_0 + \varphi_{0}$, and $\widehat X_{t_{j+1}}^{t_j,\varphi}$ denotes the one-step approximation to \eqref{eq.def:tildeX} generated by the stochastic theta method, i.e., 
\begin{align} 
\widehat X_{t_{j+1}}^{t_j,\varphi} 
&= X_0 + \varphi_{h} + \int_{t_j}^{t_{j+1}} \theta \mu(t_{j+1}, s, \widehat X_{t_{j+1}}^{t_j,\varphi}) + (1-\theta) \mu(t_{j+1}, s, \widehat X_{t_j}^{t_j,\varphi}) \rd s \notag \\ 
&\quad + \int_{t_j}^{t_{j+1}} \sigma(t_{j+1},t_j, \widehat X_{t_j}^{t_j,\varphi}) \rd W_s. \label{eq.hatX_v}
\end{align}
Note that there is also a continuous-time version for $v \in [t_j, t_{j+1}]$, 
\begin{small}
\begin{align}
\widehat X_v^{t_j, \varphi}
= X_0 + \varphi_{v-t_j} + \int_{t_j}^v \theta \mu(v, s, \widehat X_{t_{j+1}}^{t_j, \varphi}) + (1-\theta) \mu(v, s, \widehat X_{t_j}^{t_j, \varphi}) \rd s + \int_{t_j}^v \sigma(v, t_j, \widehat X_{t_j}^{t_j, \varphi}) \rd W_s. \label{eq.hatX_vcontinuous}
\end{align}
\end{small}
Let $k \in \{1,2,\cdots,N\}$. According to \eqref{eq.semigroup:tildeX} and \eqref{eq.barP_rt}, one can obtain 
\begin{align*}
\bar{Y}_{t_k}(u) 
&= \bar{P}_{h,t_{k-1}}(\bar{Y}_{t_{k-1}})_{u} \\
&= \bar{Y}_{t_{k-1}}(h+u) + \int_{t_{k-1}}^{t_{k}} \theta \mu(t_{k}+u, s, \widehat X_{t_k}^{t_{k-1},\bar{Y}_{t_{k-1}}}) + (1-\theta) \mu(t_{k}+u, s, \widehat X_{t_{k-1}}^{t_{k-1},\bar{Y}_{t_{k-1}}} ) \rd s \\
&\quad + \int_{t_{k-1}}^{t_{k}} \sigma(t_{k}+u, t_{k-1}, \widehat X_{t_{k-1}}^{t_{k-1},\bar{Y}_{t_{k-1}}} ) \rd W_s \\
&= \bar{Y}_{t_{k-1}}(h+u) + \int_{t_{k-1}}^{t_{k}} \theta \mu(t_{k}+u, s, \bar{X}_{t_{k}}) + (1-\theta) \mu(t_{k}+u, s, \bar{X}_{t_{k-1}}) \rd s \\
&\quad + \int_{t_{k-1}}^{t_{k}} \sigma(t_{k}+u, t_{k-1}, \bar{X}_{t_{k-1}}) \rd W_s, 
\end{align*}
where the facts that $\widehat X_{t_k}^{t_{k-1},\bar{Y}_{t_{k-1}}} = \bar{X}_{t_{k}}$ and $\widehat X_{t_{k-1}}^{t_{k-1},\bar{Y}_{t_{k-1}}} = \bar{X}_{t_{k-1}}$ are also used in the last step. Similarly, 
\begin{align*}
\bar{Y}_{t_{k-1}}(h+u)
&= \bar{Y}_{t_{k-2}}(2h+u) + \int_{t_{k-2}}^{t_{k-1}} \theta \mu(t_{k}+u, s, \bar{X}_{t_{k-1}}) + (1-\theta) \mu(t_{k}+u, s, \bar{X}_{t_{k-2}}) \rd s \\
&\quad + \int_{t_{k-2}}^{t_{k-1}} \sigma(t_{k}+u, t_{k-2}, \bar{X}_{t_{k-2}}) \rd W_s. 
\end{align*} 
Then, iterating the above recurrence relation with the fact $\bar{Y}_0 = \widetilde{\mathbf{0}}$ implies that for any $u \geq 0$, 
\begin{align} \label{eq.barP_tk_0}
\bar{Y}_{t_k}(u) 
= \int_0^{t_k} \theta \mu(t_k+u, s, \bar{X}_{\bar{s}} ) + (1-\theta) \mu(t_k+u, s, \bar{X}_{\underline{s}} ) \rd s + \int_0^{t_k} 
\sigma(t_k+u, \underline{s}, \bar{X}_{\underline{s}} ) \rd W_s. 
\end{align}
Finally, the scheme derived from \eqref{eq:Xtk} coincides with the stochastic theta method \eqref{eq:SthetaSVIE}.

\begin{assumption} \label{assum.theta1}
$\mu \in \cC_b^{1,0,4}([0, T] \times [0, T] \times \mathbb{R}^d; \hR^{d})$ and $\sigma \in \cC_b^{1,1,4}([0, T] \times [0, T] \times \mathbb{R}^d; \hR^{d \times m})$, where the subscript $b$ means that all derivatives of $\mu$ and $\sigma$ with respect to the third variable, up to order $4$, are bounded. Moreover, there is a constant $C := C(T) > 0$ such that for any $x \in \hR^d$ and any $t, t_1, t_2, s, s_1, s_2 \in [0,T]$, 
\begin{align}
& | \sigma (t,s_1,x) - \sigma (t,s_2,x)| 
\leq C |s_1 - s_2| (1+|x|), \label{eq:asp1} \\ 
& | \mu (t_1, s, x) - \mu (t_2, s, x)| + | \sigma (t_1, s, x) - \sigma (t_2, s, x)| 
\leq C |t_1 - t_2| (1+|x|), \label{eq:asp2} \\ 
& | \partial_2 \sigma (t_1, s, x) - \partial_2 \sigma (t_2, s, x)| + | \nabla_x \sigma (t_1, s, x) - \nabla_x \sigma (t_2, s, x)| 
\leq C |t_1 - t_2| (1+|x|). \label{eq:asp3}
\end{align}
\end{assumption}

For example, the coefficients $\mu$ and $\sigma$ given by $\mu(t,s,x) = \sigma(t,s,x) = t-s+x$ or $\mu(t,s,x) = \sigma(t,s,x) = \sin(t-s+x)$ both satisfy Assumption \ref{assum.theta1} in the scalar case. Under Assumption \ref{assum.theta1}, the following proposition can be established using standard arguments in \cite{JourdainPages2025} together with the H\"older and Burkholder--Davis--Gundy inequalities, and its proof is omitted.

\begin{proposition} 
Let $p \geq 2$ and $j \in \{0,1,\dots,N-1\}$. Let $X$ and $\bar{X}$ denote the solutions of \eqref{eq.SVIE} and \eqref{eq:SthetaSVIEcontinuous}, respectively. For $t \in [0,T]$ and $\varphi \in \Phi_{2T} \cap \mathcal{C}^1([0,\infty); \mathbb{R}^d)$, denote by $\widetilde{X}^{t,\varphi}$ and $\widehat{X}^{t_j,\varphi}$ the processes defined in \eqref{eq.def:tildeX} and \eqref{eq.hatX_vcontinuous}, respectively. Then under Assumption \ref{assum.theta1}, there exists a constant $C := C(T) > 0$, independent of $\varphi$ and $t$, such that for all $t \in [0,T]$ and $ t_j \leq v_1 \leq v \leq v_2 \leq t_{j+1}$, 
\begin{align}
& \hE|X_t|^p + \hE|\bar{X}_t|^p \leq C, \label{eq.prop1} \\
& \hE|\widetilde{X}_v^{t_j,\varphi}|^p + \hE|\widehat{X}_v^{t_j,\varphi}|^p \leq C (1 + [\varphi]_{Lip}^p), \label{eq.prop2} \\
& \hE|\widetilde{X}_{v_2}^{t_j,\varphi} - \widetilde{X}_{v_1}^{t_j,\varphi}|^p + \hE|\widehat{X}_{v_2}^{t_j,\varphi} - \widehat{X}_{v_1}^{t_j,\varphi}|^p \leq C (1 + [\varphi]_{Lip}^p) (v_2 - v_1)^{p/2}. \label{eq.prop3}
\end{align}
\end{proposition}

According to Theorem \ref{Thm.ImportantII}, the first-order weak convergence of the stochastic theta method \eqref{eq:SthetaSVIE} for SVIE \eqref{eq.SVIE} can be obtained by verifying Assumptions \ref{assum.barP is C^1} and \ref{assum.importantII} with $\nu = 4$ and $q = 1$, which are established in Lemmas \ref{lem.theta2} and \ref{lem.theta1}, respectively.

\begin{lemma} \label{lem.theta2}
Let $k \in \{0,1,\cdots,N-1\}$. If Assumption \ref{assum.theta1} holds, then there is a constant $C > 0$ independent of $k$ and $h$ such that $\bar P_{t_k, 0}(\widetilde{\mathbf{0}}) \in \cC^1 ([0, \infty); \hR^d)$ and $\mathbb{E} \left( \big[ \bar P_{t_k, 0} (\widetilde{\mathbf{0}}) \big]_{Lip}^4 \right) \leq C$. 
\end{lemma}

\begin{proof}
In view of the relation $ \bar{P}_{t_k,0}(\widetilde{\mathbf{0}}) = \bar{Y}_{t_k}$ and \eqref{eq.barP_tk_0}, we have that for any $u \geq 0$, 
\begin{align*}
\bar P_{t_k, 0}(\widetilde{\mathbf{0}})_u 
= \int_0^{t_k} \theta \mu(t_k+u, s, \bar{X}_{\bar{s}} ) + (1-\theta) \mu(t_k+u, s, \bar{X}_{\underline{s}} ) \rd s + \int_0^{t_k} 
\sigma(t_k+u, \underline{s}, \bar{X}_{\underline{s}} ) \rd W_s. 
\end{align*}
Then, it follows from Fubini's theorem that 
\begin{align*}
\bar P_{t_k, 0}(\widetilde{\mathbf{0}})_u 
&= \int_0^{t_k} \int_0^u \theta \partial_1 \mu(t_k+v, s, \bar{X}_{\bar{s}} ) + (1-\theta) \partial_1 \mu(t_k+v, s, \bar{X}_{\underline{s}} ) \rd v \rd s \\
&\quad + \int_0^{t_k} \theta \mu(t_k, s, \bar{X}_{\bar{s}} ) + (1-\theta) \mu(t_k, s, \bar{X}_{\underline{s}} ) \rd s \\
&\quad + \int_0^{t_k} \int_0^u \partial_1 \sigma(t_k+v, \underline{s}, \bar{X}_{\underline{s}} ) \rd v \rd W_s + \int_0^{t_k} \sigma(t_k, \underline{s}, \bar{X}_{\underline{s}} ) \rd W_s \\
&= \int_0^u \int_0^{t_k} \theta \partial_1 \mu(t_k+v, s, \bar{X}_{\bar{s}} ) + (1-\theta) \partial_1 \mu(t_k+v, s, \bar{X}_{\underline{s}} ) \rd s \rd v \\
&\quad + \int_0^{t_k} \theta \mu(t_k, s, \bar{X}_{\bar{s}} ) + (1-\theta) \mu(t_k, s, \bar{X}_{\underline{s}} ) \rd s \\
&\quad + \int_0^u \int_0^{t_k} \partial_1 \sigma(t_k+v, \underline{s}, \bar{X}_{\underline{s}} ) \rd W_s \rd v + \int_0^{t_k} \sigma(t_k, \underline{s}, \bar{X}_{\underline{s}} ) \rd W_s,
\end{align*}
which implies that $\bar P_{t_k, 0}(\widetilde{\mathbf{0}}) \in \cC^1 ([0, \infty); \hR^d)$ and 
\begin{align*}
\frac{\rd}{\rd u}\bar P_{t_k, 0}(\widetilde{\mathbf{0}})_u
&= \int_0^{t_k} \theta \partial_1 \mu(t_k+u, s, \bar{X}_{\bar{s}} ) + (1-\theta) \partial_1 \mu(t_k+u, s, \bar{X}_{\underline{s}} ) \rd s \\
&\quad + \int_0^{t_k} \partial_1 \sigma(t_k+u, \underline{s}, \bar{X}_{\underline{s}} ) \rd W_s.
\end{align*}
Hence, using Assumption \ref{assum.theta1}, H\"older's inequality, the Burkholder--Davis--Gundy inequality, and \eqref{eq.prop1} yields 
\begin{align*}
\mathbb{E} \left( \big[ \bar P_{t_k, 0} (\widetilde{\mathbf{0}}) \big]_{Lip}^4 \right)
= \hE \bigg[ \Big( \sup_{u \geq 0} \Big| \frac{\rd}{\rd u}\bar P_{t_k, 0}(\widetilde{\mathbf{0}})_u \Big| \Big)^4 \bigg] 
\leq C.
\end{align*}
The proof is completed. 
\end{proof}

\begin{lemma} \label{lem.theta1}
Let $g \in \cC_b^{4} ( \Phi_{2T}; \hR)$. Assume that Assumption \ref{assum.theta1} holds. Then there exists a constant $C$ such that for any $k \in \{ 0, 1, \dots, N-1 \}$, $\varphi \in \Phi_{2T}\cap\cC^1 ([0, \infty); \hR^d)$, $m \in \{1, 2, 3\}$, and $\xi \in [0, 1],$
\begin{align}
&\left| \hE \left\< \nabla^m g( \tau_h \varphi ), (P_{h,t_k}(\varphi)- \tau_h \varphi )^{\otimes m} - (\bar P_{h,t_k}(\varphi)- \tau_h \varphi )^{\otimes m} \right\> \right| 
\leq C(1 + [\varphi]_{Lip}^4)h^{2}, \label{eq.2order1} \\
&\left| \hE \left\<\nabla^{4} g( \tau_h \varphi + \xi (P_{h,t_k}(\varphi)- \tau_h \varphi )), (P_{h,t_k}(\varphi)- \tau_h \varphi )^{\otimes 4} \right\> \right| 
\leq C(1 + [\varphi]_{Lip}^4)h^{2}, \label{eq.2order2} \\
&\left| \hE \left\<\nabla^{4} g( \tau_h \varphi + \xi (\bar P_{h,t_k}(\varphi)- \tau_h \varphi )), (\bar P_{h,t_k}(\varphi)- \tau_h \varphi )^{\otimes 4} \right\> \right| 
\leq C(1 + [\varphi]_{Lip}^4)h^{2}, \label{eq.2order3} 
\end{align}
where $C > 0$ depends on $g$ only through its semi-norm $\|g\|_{\cC_b^4 ( \Phi_{2T})}$, but is independent of $\varphi$ and $h$. 
\end{lemma}

\begin{proof}
Throughout this proof, we denote $\widetilde{X}:=\widetilde{X}^{t_k,\varphi}$ and $\widehat{X}:=\widehat{X}^{t_k,\varphi}$ for short. It follows from \eqref{eq.def:tildeX} that for $v \in [t_k, t_{k+1} )$ and sufficiently small $\varepsilon > 0$, 
\begin{align*}
\widetilde{X}_{v+\varepsilon}-\widetilde{X}_v 
&= \varphi_{v+\varepsilon -t_k}-\varphi_{v-t_k}+ \int_v^{v+\varepsilon} \mu(v+\varepsilon, s, \widetilde{X}_s) \rd s + \int_{t_k}^v \mu(v+\varepsilon, s, \widetilde{X}_s) - \mu(v, s, \widetilde{X}_s) \rd s \\ 
&\quad + \int_v^{v+\varepsilon} \sigma(v+\varepsilon, s, \widetilde{X}_s) \rd W_s + \int_{t_k}^v \sigma(v+\varepsilon, s, \widetilde{X}_s) - \sigma(v, s, \widetilde{X}_s) \rd W_s, 
\end{align*}
which implies the differential 
\begin{align} 
\rd \widetilde{X}_v 
&= \dot{\varphi}_{v-t_k} \rd v + \mu(v, v, \widetilde{X}_v) \rd v + \left( \int_{t_k}^v \partial_1 \mu(v, s, \widetilde{X}_s) \rd s \right) \rd v \notag \\ 
&\quad + \sigma(v, v, \widetilde{X}_v) \rd W_v + \left( \int_{t_k}^v \partial_1 \sigma(v, s, \widetilde{X}_s) \rd W_s \right) \rd v \label{eq.dtildeX}. 
\end{align}
Applying the classical It\^o formula to $F \in \cC^2 (\hR^d; \hR)$ yields that for $v \in [t_k, t_{k+1} )$, 
\begin{align} 
\rd F (\widetilde{X}_v)
&= \nabla F( \widetilde{X}_v) \big( \dot{\varphi}_{v-t_k} + \mu(v, v, \widetilde{X}_v) \big) \rd v + \nabla F ( \widetilde{X}_v) \sigma(v, v, \widetilde{X}_v) \rd W_v \nonumber \\ 
&\quad + \nabla F (\widetilde{X}_v) \left( \int_{t_k}^v \partial_1 \mu(v, s, \widetilde{X}_s) \rd s + \int_{t_k}^v \partial_1 \sigma(v, s, \widetilde{X}_s) \rd W_s \right) \rd v \nonumber \\ 
&\quad + \frac{1}{2} \mbox{trace} \big( \sigma(v, v, \widetilde{X}_v)^{\top} \nabla^2F (\widetilde{X}_v) \sigma(v, v, \widetilde{X}_v) \big) \rd v. \label{eq.dpsi}
\end{align}

\textit{Proof of \eqref{eq.2order1} with $m = 1$}. Using \eqref{eq.def:P_rt} and \eqref{eq.barP_rt} shows 
\begin{align*}
\hE [ P_{h,t_k}(\varphi)- \tau_h \varphi ] 
&= \hE \int_{t_k}^{t_{k+1}} \mu( t_{k+1}+\cdot, s, \widetilde{X}_s) \rd s, \\ 
\hE [\bar P_{h,t_k}(\varphi)- \tau_h \varphi ] 
&= \hE \int_{t_k}^{t_{k+1}} \theta \mu(t_{k+1}+\cdot, s, \widehat X_{t_{k+1}} ) + (1-\theta) \mu(t_{k+1}+\cdot, s, \widehat X_{t_k} ) \rd s, 
\end{align*} 
which implies 
\begin{align}
&\quad\, \left | \hE \left[ \left\< \nabla g( \tau_h \varphi ), (P_{h,t_k}(\varphi)- \tau_h \varphi ) - (\bar P_{h,t_k}(\varphi)- \tau_h \varphi ) \right\> \right] \right | \notag \\
&\leq \left| \int_{t_k}^{t_{k+1}} \hE \big\< \nabla g( \tau_h \varphi ), \theta \big( \mu(t_{k+1}+\cdot, s, \widetilde X_{s}) - \mu(t_{k+1}+\cdot, s, \widehat X_{t_{k+1}}) \big) \big\> \rd s \right| \notag \\
&\quad + \left| \int_{t_k}^{t_{k+1}} \hE \big\< \nabla g( \tau_h \varphi ), (1-\theta) \big( \mu(t_{k+1}+\cdot, s, \widetilde X_{s}) - \mu(t_{k+1}+\cdot, s, \widehat X_{t_k}) \big) \big\> \rd s \right| \notag \\
&=: (\mathrm{I})_1 + (\mathrm{I})_2. \label{eq.I1I2}
\end{align}

For $(\mathrm{I})_1$, one can obtain 
\begin{small}
\begin{align*}
&\quad\ (\mathrm{I})_1 \\ 
&\leq \left| \int_{t_k}^{t_{k+1}} \hE \< \nabla g( \tau_h \varphi ), \mu(t_{k+1}+\cdot, s, \widetilde X_{s}) - \mu(t_{k+1}+\cdot, s, \widehat X_{t_k}) \> \rd s \right| \\
& + \left| \int_{t_k}^{t_{k+1}} \hE \< \nabla g( \tau_h \varphi ), \nabla_x \mu(t_{k+1}+\cdot, s, \widehat X_{t_k}) (\widehat X_{t_{k+1}} - \widehat X_{t_k}) \> \rd s \right| \\
& + \left| \int_{t_k}^{t_{k+1}} \hE \left\< \nabla g( \tau_h \varphi ), \int_0^1 \<\nabla_{xx} \mu(t_{k+1}+\cdot, s, \widehat X_{t_k} + \xi (\widehat X_{t_{k+1}} - \widehat X_{t_k}) ), (\widehat X_{t_{k+1}} - \widehat X_{t_k})^{\otimes 2} \> (1-\xi) \rd \xi 
\right\> \rd s \right| \\
&=: (\mathrm{I})_{1,1} + (\mathrm{I})_{1,2} + (\mathrm{I})_{1,3}, 
\end{align*} 
\end{small}
where we used the Taylor expansion 
\begin{align*}
&\quad\ \mu(t_{k+1}+\cdot, s, \widehat X_{t_{k+1}}) \\
&= \mu(t_{k+1}+\cdot, s, \widehat X_{t_k}) + \nabla_x \mu(t_{k+1}+\cdot, s, \widehat X_{t_k}) (\widehat X_{t_{k+1}} - \widehat X_{t_k}) \\
&\quad + \int_0^1 \<\nabla_{xx} \mu(t_{k+1}+\cdot, s, \widehat X_{t_k} + \xi (\widehat X_{t_{k+1}} - \widehat X_{t_k}) ), (\widehat X_{t_{k+1}} - \widehat X_{t_k})^{\otimes 2} \> (1-\xi) \rd \xi. 
\end{align*} 
To estimate $(\mathrm{I})_{1,1}$, we define $G(x) := \< \nabla g( \tau_h \varphi ), \mu(t_{k+1}+\cdot, s, x) \>$ for $x \in\mathbb{R}^d$. Then $G \in \cC_b^2 (\hR^d; \hR)$ and 
\begin{align*}
& \nabla G(x) (y_1) = \< \nabla g( \tau_h \varphi ), \nabla_x \mu(t_{k+1}+\cdot, s, x)y_1 \>,\qquad y_1 \in \mathbb{R}^d, \\
& \nabla^2 G(x) (y_1, y_2) = \< \nabla g( \tau_h \varphi ), \<\nabla_{xx} \mu(t_{k+1}+\cdot, s, x),y_1\otimes y_2 \>\>,\qquad y_1, y_2 \in \mathbb{R}^d.
\end{align*}
In view of \eqref{eq.dpsi}, one has that for $\alpha \in [t_k, t_{k+1} )$, 
\begin{align*} 
\rd G (\widetilde{X}_\alpha)
&= \nabla G ( \widetilde{X}_\alpha) \big(\dot{\varphi}_{\alpha-t_k} + \mu(\alpha, \alpha, \widetilde{X}_\alpha) \big) \rd \alpha + \nabla G ( \widetilde{X}_\alpha) \sigma(\alpha, \alpha, \widetilde{X}_\alpha) \rd W_\alpha \\ 
&\quad + \nabla G (\widetilde{X}_\alpha) \left( \int_{t_k}^\alpha \partial_1 \mu(\alpha, s, \widetilde{X}_s) \rd s + \int_{t_k}^\alpha \partial_1 \sigma(\alpha, s, \widetilde{X}_s) \rd W_s \right) \rd \alpha \\ 
&\quad + \frac{1}{2} \mbox{trace} \big( \sigma (\alpha, \alpha, \widetilde{X}_\alpha)^{\top} \nabla^2 G (\widetilde{X}_\alpha) \sigma (\alpha, \alpha, \widetilde{X}_\alpha) \big) \rd \alpha. 
\end{align*}
From \eqref{eq:asp2}, one obtains that for any $x \in \hR^d$ and any $t, s \in [0,T]$,
\begin{align} \label{linear growth}
|\partial_1 \mu (t,s,x)| + |\partial_1 \sigma (t,s,x)| \leq C(1+|x|).
\end{align}
Since $G \in \cC_b^2 (\hR^d; \hR)$, $\varphi$ is Lipschitz continuous, together with Assumption \ref{assum.theta1}, \eqref{linear growth}, and \eqref{eq.prop2}, it holds that
\begin{align*}
\left| \hE \left[ \int_{t_k}^s \rd G (\widetilde{X}_\alpha) \right] \right| 
&\leq \hE \left| \int_{t_k}^s \nabla G ( \widetilde{X}_\alpha) \big(\dot{\varphi}_{\alpha-t_k} + \mu(\alpha, \alpha, \widetilde{X}_\alpha) \big) \rd \alpha \right| \\
&\quad + \hE \left| \int_{t_k}^s \nabla G (\widetilde{X}_\alpha) \left( \int_{t_k}^\alpha \partial_1 \mu(\alpha, s, \widetilde{X}_s) \rd s + \int_{t_k}^\alpha \partial_1 \sigma(\alpha, s, \widetilde{X}_s) \rd W_s \right) \rd \alpha \right| \\
&\quad + \hE \left| \int_{t_k}^s \mbox{trace} \big( \sigma (\alpha, \alpha, \widetilde{X}_\alpha)^{\top} \nabla^2 G (\widetilde{X}_\alpha) \sigma (\alpha, \alpha, \widetilde{X}_\alpha) \big) \rd \alpha \right| \\
&\leq C \| \nabla G \|_\infty \int_{t_k}^s [\varphi]_{Lip} + \| \nabla_x \mu \|_\infty (1 + \hE | \widetilde{X}_\alpha | ) \rd \alpha \\
&\quad + C \| \nabla G \|_\infty \int_{t_k}^s \int_{t_k}^\alpha (1 + \hE | \widetilde{X}_s | ) \rd s \rd \alpha \\
&\quad + C \| \nabla G \|_\infty \int_{t_k}^s \left( \int_{t_k}^\alpha (1 + \hE | \widetilde{X}_s |^2 ) \rd s \right)^\frac{1}{2} \rd \alpha \\ 
&\quad + C \| \nabla^2 G\|_\infty \int_{t_k}^s \| \nabla_x \sigma \|_\infty^2 (1 + \hE |\widetilde{X}_\alpha|^2 ) \rd \alpha \\
&\leq C (1+[\varphi]_{Lip}^2)(s-t_k).
\end{align*}
Then using the fact $\widetilde X_{t_{k}} = \widehat X_{t_{k}} = X_0 + \varphi_0$ indicates 
\begin{align} \label{eq.A1bound}
(\mathrm{I})_{1,1} = \left| \int_{t_k}^{t_{k+1}} \hE \left[ \int_{t_k}^s \rd G (\widetilde{X}_\alpha) \right] \rd s \right| \leq C(1+[\varphi]_{Lip}^2)h^2.
\end{align}

As for $(\mathrm{I})_{1,2}$, it follows from \eqref{eq.hatX_v} that 
\begin{align*}
\widehat X_{t_{k+1}}-\widehat X_{t_k} &= \varphi_h-\varphi_0+ \int_{t_k}^{t_{k+1}} \theta \mu(t_{k+1}, s, \widehat X_{t_{k+1}}) \rd s \\
&\quad\ + \int_{t_k}^{t_{k+1}} (1-\theta) \mu(t_{k+1}, s, \widehat X_{t_k}) \rd s + \int_{t_k}^{t_{k+1}} \sigma(t_{k+1}, {t_k}, \widehat X_{t_k}) \rd W_s, 
\end{align*}
which implies 
\begin{align*}
(\mathrm{I})_{1,2} 
&\le \int_{t_k}^{t_{k+1}} \left| \mathbb{E} \left\< \nabla g( \tau_h \varphi ), \nabla_x \mu(t_{k+1}+\cdot, s, \widehat X_{t_k}) (\varphi_h-\varphi_0) \right\> \right| \rd s \\
&\quad\ + \int_{t_k}^{t_{k+1}} \left| \mathbb{E} \left\< \nabla g( \tau_h \varphi ), \nabla_x \mu(t_{k+1}+\cdot, s, \widehat X_{t_k}) \int_{t_k}^{t_{k+1}} \theta \mu(t_{k+1}, v, \widehat X_{t_{k+1}}) \rd v \right\> \right| \rd s \\
&\quad\ + \int_{t_k}^{t_{k+1}} \left| \mathbb{E} \left\< \nabla g( \tau_h \varphi ), \nabla_x \mu(t_{k+1}+\cdot, s, \widehat X_{t_k}) \int_{t_k}^{t_{k+1}} (1-\theta) \mu(t_{k+1}, v, \widehat X_{t_k}) \rd v \right\> \right| \rd s \\
&\quad\ + \int_{t_k}^{t_{k+1}} \left| \hE \left\< \nabla g( \tau_h \varphi ), \nabla_x \mu(t_{k+1}+\cdot, s, \widehat X_{t_k}) \sigma(t_{k+1}, t_k, \widehat X_{{t_k}}) \big( W(t_{k+1}) - W(t_k) \big) \right\> \right| \rd s \\
&=: (\mathrm{I})_{1,2,1} + (\mathrm{I})_{1,2,2} + (\mathrm{I})_{1,2,3} + (\mathrm{I})_{1,2,4}. 
\end{align*}
Since $\nabla g$ and $\nabla_x \mu$ are bounded, and $\varphi$ is Lipschitz continuous, one deduces from \eqref{eq.prop2} that 
\begin{align*}
(\mathrm{I})_{1,2,1} + (\mathrm{I})_{1,2,2} + (\mathrm{I})_{1,2,3} \le C (1+[\varphi]_{Lip}) h^2.
\end{align*}
Note that $(\mathrm{I})_{1,2,4}=0$. Thus, one gets $(\mathrm{I})_{1,2} \leq C (1+[\varphi]_{Lip})h^2$.

For $(\mathrm{I})_{1,3}$, in view of the boundedness of $\nabla g$ and $\nabla_{xx} \mu$, as well as \eqref{eq.prop3}, 
\begin{align*}
(\mathrm{I})_{1,3} \leq C \int_{t_k}^{t_{k+1}} \hE [|\widehat X_{t_{k+1}}-\widehat X_{t_k}|^2] \rd s \leq C (1+[\varphi]_{Lip}^2) h^2.
\end{align*}
Combining the above estimates, one concludes $(\mathrm{I})_1 \leq C(1+[\varphi]_{Lip}^2 ) h^2$. Moreover, it follows from \eqref{eq.A1bound} that $(\mathrm{I})_2 = (1-\theta) (I)_{1,1} \leq C (1+[\varphi]_{Lip}^2) h^2$. Thus, recalling \eqref{eq.I1I2} yields 
\begin{align*}
\left| \hE \left[ \left\< \nabla g( \tau_h \varphi ), (P_{h,t_k}(\varphi)- \tau_h \varphi ) - (\bar P_{h,t_k}(\varphi)- \tau_h \varphi ) \right\> \right] \right| \leq C (1+[\varphi]_{Lip}^2 ) h^2.
\end{align*}

\textit{Proof of \eqref{eq.2order1} with $m = 2$}. 
Using \eqref{eq.def:P_rt}, \eqref{eq.barP_rt}, and the boundedness of $\nabla^2 g$ shows 
\begin{align} \label{eq.m=2}
\left | \hE \left\< \nabla^2 g( \tau_h \varphi ), (P_{h,t_k}(\varphi)- \tau_h \varphi )^{\otimes 2} - (\bar P_{h,t_k}(\varphi)- \tau_h \varphi )^{\otimes 2} \right\> \right | 
\leq \sum_{i=1}^4 (\mathrm{II})_{i} 
\end{align}
with
\begin{small}
\begin{align*}
& (\mathrm{II})_{1} := C \hE \bigg[ \left\| \int_{t_k}^{t_{k+1}} \mu( t_{k+1}+\cdot, s, \widetilde{X}_s) \rd s \right\|_{L^ \infty}^2 \\
&\qquad\quad + \left\| \int_{t_k}^{t_{k+1}} \theta \mu(t_{k+1}+\cdot, s, \widehat X_{t_{k+1}} ) + (1-\theta) \mu(t_{k+1}+\cdot, s, \widehat X_{t_k} ) \rd s \right\|_{L^ \infty}^2 \bigg], \\
& (\mathrm{II})_{2} := C \hE \left[ \left\| \int_{t_k}^{t_{k+1}} \mu( t_{k+1}+\cdot, s, \widetilde{X}_s) \rd s \right\|_{L^ \infty} \left\| \int_{t_k}^{t_{k+1}} \sigma( t_{k+1}+\cdot, s, \widetilde{X}_s) - \sigma( t_{k+1}+\cdot, t_k, \widetilde{X}_{t_k}) \rd W_s \right\|_{L^ \infty} \right], \\
& (\mathrm{II})_{3} := C \hE \left[ \left\| \int_{t_k}^{t_{k+1}} \mu( t_{k+1}+\cdot, s, \widetilde{X}_s) - \theta \mu(t_{k+1}+\cdot, s, \widehat X_{t_{k+1}} ) - (1-\theta) \mu(t_{k+1}+\cdot, s, \widehat X_{t_k}) \rd s \right\|_{L^ \infty} \right.\\
&\qquad\qquad\quad \times \left. \left\| \int_{t_k}^{t_{k+1}} \sigma( t_{k+1}+\cdot, t_k, \widehat{X}_{t_k}) \rd W_s \right\|_{L^ \infty} \right], \\
& (\mathrm{II})_{4} := \left| \hE \left\< \nabla^2 g( \tau_h \varphi ), \left( \int_{t_k}^{t_{k+1}} \sigma( t_{k+1}+\cdot, s, \widetilde{X}_s) \rd W_s\right)^{\otimes 2} - \left( \int_{t_k}^{t_{k+1}} 
 \sigma(t_{k+1}+\cdot, t_k, \widehat X_{t_k} ) \rd W_s\right)^{\otimes 2} \right\> \right|.
\end{align*}
\end{small}
Here and after, the notation $\| \varphi \|_{L^ \infty} := \sup_{u\geq0}|\varphi_u|$ is used for $\varphi \in \Phi_{2T}$.

By \eqref{eq.Appendix lem2_1} and \eqref{eq.prop2}, one has 
\begin{align} \label{eq.m=2_1}
(\mathrm{II})_1 \leq C(1+[\varphi]_{Lip}^2) h^2. 
\end{align} 
Applying Lemma \ref{Appendix lem3}, \eqref{eq.prop2}, and \eqref{eq.prop3}, one obtains
\begin{align*}
&\quad\ \hE \left\| \int_{t_k}^{t_{k+1}} \sigma( t_{k+1}+\cdot, s, \widetilde{X}_s) - \sigma( t_{k+1}+\cdot, t_k, \widetilde{X}_{t_k}) \rd W_s \right\|_{L^ \infty}^2 \\
&\leq C \hE \left\| \int_{t_k}^{t_{k+1}} \sigma( t_{k+1}+\cdot, s, \widetilde{X}_s) - \sigma( t_{k+1}+\cdot, t_k, \widetilde{X}_s) \rd W_s \right\|_{L^ \infty}^2 \\
&\quad + C \hE \left\| \int_{t_k}^{t_{k+1}} \sigma( t_{k+1}+\cdot, t_k, \widetilde{X}_s) - \sigma( t_{k+1}+\cdot, t_k, \widetilde{X}_{t_k}) \rd W_s \right\|_{L^ \infty}^2 \\
&\leq C(1+[\varphi]_{Lip}^2) h^3 + C(1+[\varphi]_{Lip}^4) h^2 \\
&\leq C(1+[\varphi]_{Lip}^4) h^2.
\end{align*}
Hence, by the Cauchy--Schwarz inequality and \eqref{eq.Appendix lem2_1}, 
\begin{align} \label{eq.m=2_2}
(\mathrm{II})_2 \leq C(1+[\varphi]_{Lip}^4) h^2. 
\end{align}
Using the relation $\widetilde X_{t_{k}} = \widehat X_{t_{k}} = X_0 + \varphi_0$ together with \eqref{eq.prop3}, one obtains that for $s \in [t_k, t_{k+1}]$,
\begin{align} \label{eq.strong order}
\hE |\widetilde{X}_{s}-\widehat{X}_{t_{k+1}}|^2 \leq 2 \hE |\widetilde{X}_{s}-\widetilde{X}_{t_{k}}|^2 + 2 \hE |\widehat{X}_{t_k}-\widehat{X}_{t_{k+1}}|^2 \leq C (1 + [\varphi]_{Lip}^2) h.
\end{align}
A similar analysis to that in \eqref{eq.Appendix lem2_1}, combined with the Cauchy--Schwarz inequality, \eqref{eq.Appendix lem2_2}, \eqref{eq.prop2}, \eqref{eq.prop3}, and \eqref{eq.strong order}, one has
\begin{align} \label{eq.m=2_3}
(\mathrm{II})_{3} \leq C(1+[\varphi]_{Lip}^2)h^2.
\end{align} 
For $(\mathrm{II})_{4}$, using the It\^o isometry implies 
\begin{align*}
(\mathrm{II})_{4} 
&= \left| \int_{t_k}^{t_{k+1}} \hE \< \nabla^2 g( \tau_h \varphi ), ( \sigma (t_{k+1}+\cdot, s, \widetilde X_s))^{\otimes 2} - ( \sigma(t_{k+1}+\cdot, t_k, \widehat X_{t_k}) )^{\otimes 2} \> \rd s \right| \\
&\leq \left| \int_{t_k}^{t_{k+1}} \hE \< \nabla^2 g( \tau_h \varphi ), ( \sigma(t_{k+1}+\cdot, s, \widetilde X_{s}))^{\otimes 2} - ( \sigma(t_{k+1}+\cdot, t_k, \widetilde X_{s}))^{\otimes 2} \> \rd s \right| \\
&\quad + \left| \int_{t_k}^{t_{k+1}} \hE \< \nabla^2 g( \tau_h \varphi ), ( \sigma(t_{k+1}+\cdot, t_k, \widetilde X_{s}) )^{\otimes 2} - ( \sigma (t_{k+1}+\cdot, t_k, \widehat X_{t_k}) )^{\otimes 2} \> \rd s \right| \\
&=: (\mathrm{II})_{4,1} +(\mathrm{II})_{4,2}.
\end{align*} 
For $(\mathrm{II})_{4,1}$, we define $G_1(r) := \< \nabla^2 g( \tau_h \varphi ), ( \sigma (t_{k+1}+\cdot, r, \widetilde X_s))^{\otimes 2} \>$ for $r \in [t_k, t_{k+1}]$. Then, it follows from \eqref{eq:asp1} and \eqref{eq.prop2} that
\begin{align*}
(\mathrm{II})_{4,1} = \left| \int_{t_k}^{t_{k+1}} \hE \left[ G_1(s) - G_1(t_k) \right] \rd s \right| \leq C(1+[\varphi]_{Lip}^2)h^2.
\end{align*}
Following the same step as in \eqref{eq.A1bound}, we obtain $(\mathrm{II})_{4,2} \leq C(1+[\varphi]_{Lip}^2 ) h^2$. Thus, 
\begin{align} \label{eq.m=2_4}
(\mathrm{II})_{4} \leq C(1+[\varphi]_{Lip}^2 ) h^2.
\end{align}
Combining the estimates \eqref{eq.m=2}, \eqref{eq.m=2_1}, \eqref{eq.m=2_2}, \eqref{eq.m=2_3}, and \eqref{eq.m=2_4}, we complete the proof of \eqref{eq.2order1} in the case $m=2$.

\textit{Proof of \eqref{eq.2order1} with $m = 3$}. Using \eqref{eq.def:P_rt} and \eqref{eq.barP_rt}, together with the assumption that $\nabla^3 g$ is bounded, one has 
\begin{align} \label{eq.m=3}
\left | \hE \left\< \nabla^3 g( \tau_h \varphi ), (P_{h,t_k}(\varphi) - \tau_h \varphi )^{\otimes 3} - (\bar P_{h,t_k}(\varphi)- \tau_h \varphi )^{\otimes 3} \right\> \right | 
\leq \sum_{i=1}^4 (\mathrm{III})_{i} 
\end{align}
with
\begin{small}
\begin{align*}
&(\mathrm{III})_{1} := C \hE \left\| \int_{t_k}^{t_{k+1}} \mu( t_{k+1}+\cdot, s, \widetilde{X}_s) \rd s \right\|_{L^ \infty}^3 \\
&\qquad\qquad + C \hE \left\| \int_{t_k}^{t_{k+1}} \theta \mu(t_{k+1}+\cdot, s, \widehat X_{t_{k+1}} ) + (1-\theta) \mu(t_{k+1}+\cdot, s, \widehat X_{t_k} ) \rd s \right\|_{L^ \infty}^3, \\
&(\mathrm{III})_{2} := C \hE \bigg[ \left\| \int_{t_k}^{t_{k+1}} \mu( t_{k+1}+\cdot, s, \widetilde{X}_s) \rd s \right\|_{L^ \infty}^2 \left\| \int_{t_k}^{t_{k+1}} \sigma( t_{k+1}+\cdot, s, \widetilde{X}_s) \rd W_s \right\|_{L^ \infty} \bigg] \\ 
&\qquad\qquad + C \hE \bigg[ \left\| \int_{t_k}^{t_{k+1}} \theta \mu(t_{k+1}+\cdot, s, \widehat X_{t_{k+1}} ) + (1-\theta) \mu(t_{k+1}+\cdot, s, \widehat X_{t_k} ) \rd s \right\|_{L^ \infty}^2 \\
&\qquad\qquad\quad \times \left\| \int_{t_k}^{t_{k+1}} \sigma(t_{k+1}+\cdot, t_k, \widehat X_{t_k} ) \rd W_s \right\|_{L^ \infty} \bigg], \\
&(\mathrm{III})_{3} := C \hE \bigg[ \left\| \int_{t_k}^{t_{k+1}} \mu( t_{k+1}+\cdot, s, \widetilde{X}_s) \rd s \right\|_{L^ \infty} \left\| \int_{t_k}^{t_{k+1}} \sigma( t_{k+1}+\cdot, s, \widetilde{X}_s) \rd W_s \right\|_{L^ \infty}^2 \bigg] \\ 
&\qquad\qquad + C \hE \bigg[ \left\| \int_{t_k}^{t_{k+1}} \theta \mu(t_{k+1}+\cdot, s, \widehat X_{t_{k+1}} ) + (1-\theta) \mu(t_{k+1}+\cdot, s, \widehat X_{t_k} ) \rd s \right\|_{L^ \infty} \\
&\qquad\qquad\quad \times \left\| \int_{t_k}^{t_{k+1}} \sigma(t_{k+1}+\cdot, t_k, \widehat X_{t_k} ) \rd W_s \right\|_{L^ \infty}^2 \bigg], \\
&(\mathrm{III})_{4} := \left| \hE \left\< \nabla^3 g( \tau_h \varphi ), \left( \int_{t_k}^{t_{k+1}} \sigma( t_{k+1}+\cdot, s, \widetilde{X}_s) \rd W_s\right)^{\otimes 3} - \left( \int_{t_k}^{t_{k+1}} 
 \sigma(t_{k+1}+\cdot, t_k, \widehat X_{t_k} ) \rd W_s\right)^{\otimes 3} \right\> \right|. 
\end{align*}
\end{small}
By the Cauchy--Schwarz inequality and Lemma \ref{Appendix lem2}, combined with \eqref{eq.prop2}, one has
\begin{align} \label{eq.m=3_1}
(\mathrm{III})_{1} + (\mathrm{III})_{2} + (\mathrm{III})_{3} \leq C(1+[\varphi]_{Lip}^3)h^2.
\end{align}

To facilitate the analysis of $(\mathrm{III})_{4}$, we define
\begin{align*}
\widetilde H(u) := \int_{t_k}^{t_{k+1}} \sigma (t_{k+1}+u, s, \widetilde{X}_{s}) \rd W_{s}, 
\qquad
\widehat H(u) := \int_{t_k}^{t_{k+1}} \sigma (t_{k+1}+u, t_k, \widehat{X}_{t_k}) \rd W_{s},
\qquad
u \geq 0.
\end{align*}
Then, using the boundedness of $\nabla^3 g$ and the Cauchy--Schwarz inequality indicates 
\begin{align*}
(\mathrm{III})_{4} &\leq | \hE \< \nabla^3 g(\varphi), \widetilde{H}^{\otimes 2} \otimes ( \widetilde{H} - \widehat{H} ) \> | + | \hE \< \nabla^3 g(\varphi), \widetilde{H} \otimes ( \widetilde{H} - \widehat{H} ) \otimes \widehat{H} \> | \\
&\quad + | \hE \< \nabla^3 g(\varphi), ( \widetilde{H} - \widehat{H} ) \otimes \widehat{H}^{\otimes 2} \> | \\
&\leq C \hE \big[ \| \widetilde{H} \|_{L^ \infty}^2 \| \widetilde{H}-\widehat{H} \|_{L^ \infty} \big] + C \hE \big[\| \widetilde{H} \|_{L^ \infty} \| \widetilde{H}-\widehat{H} \|_{L^ \infty} \| \widehat{H} \|_{L^ \infty} \big] \\
&\quad + C \hE \big[ \| \widetilde{H}-\widehat{H} \|_{L^ \infty} \| \widehat{H} \|_{L^ \infty}^2 \big] \\
&\leq C ( \hE \| \widetilde{H} \|_{L^ \infty}^4 )^{1/2} ( \hE \| \widetilde{H}-\widehat{H} \|_{L^ \infty}^2 )^{1/2} + C ( \hE \| \widetilde{H} \|_{L^ \infty}^4 )^{1/4} ( \hE \| \widetilde{H}-\widehat{H} \|_{L^ \infty}^2 )^{1/2} ( \hE \| \widehat{H} \|_{L^ \infty}^4 )^{1/4} \\
&\quad + C ( \hE \| \widetilde{H}-\widehat{H} \|_{L^ \infty}^2 )^{1/2} ( \hE \| \widehat{H} \|_{L^ \infty}^4 )^{1/2}.
\end{align*} 
Applying Lemma \ref{Appendix lem3} and \eqref{eq.Appendix lem2_2} together with \eqref{eq.prop2} and \eqref{eq.prop3}, one obtains 
\begin{align*}
\hE \| \widetilde{H} \|_{L^ \infty}^4 + \hE \| \widehat{H} \|_{L^ \infty}^4 + \hE \| \widetilde{H}-\widehat{H} \|_{L^ \infty}^2 \leq C (1+[\varphi]_{Lip}^4) h^2. 
\end{align*}
Hence, it holds that
\begin{align} \label{eq.m=3_2}
(\mathrm{III})_{4} \leq C(1+[\varphi]_{Lip}^4)h^2.
\end{align}
Combining the estimates \eqref{eq.m=3}, \eqref{eq.m=3_1}, and \eqref{eq.m=3_2}, we complete the proof of \eqref{eq.2order1} in the case $m=3$.

\textit{Proof of \eqref{eq.2order2}}. By \eqref{eq.def:P_rt} and the boundedness of $\nabla^4 g$, together with the Cauchy--Schwarz inequality, Lemma \ref{Appendix lem2}, and \eqref{eq.prop2}, one gets 
\begin{align*}
&\quad\, \left| \hE \left\<\nabla^4 g( \tau_h \varphi + \eta (P_{h,t_k}(\varphi)- \tau_h \varphi )), (P_{h,t_k}(\varphi)- \tau_h \varphi )^{\otimes 4} \right\> \right| \\
& \leq C \hE \left[ \left\| \int_{t_k}^{t_{k+1}} \mu( t_{k+1}+\cdot, s, \widetilde{X}_s) \rd s \right\|_{L^ \infty}^4 \right] \\
&\quad + C \hE \left[ \left\| \int_{t_k}^{t_{k+1}} \mu( t_{k+1}+\cdot, s, \widetilde{X}_s) \rd s \right\|_{L^ \infty}^3 \left\| \int_{t_k}^{t_{k+1}} \sigma( t_{k+1}+\cdot, s, \widetilde{X}_s) \rd W_s \right\|_{L^ \infty} \right] \\
&\quad + C \hE \left[ \left\| \int_{t_k}^{t_{k+1}} \mu( t_{k+1}+\cdot, s, \widetilde{X}_s) \rd s \right\|_{L^ \infty}^2 \left\| \int_{t_k}^{t_{k+1}} \sigma( t_{k+1}+\cdot, s, \widetilde{X}_s) \rd W_s \right\|_{L^ \infty}^2 \right] \\
&\quad + C \hE \left[ \left\| \int_{t_k}^{t_{k+1}} \mu( t_{k+1}+\cdot, s, \widetilde{X}_s) \rd s \right\|_{L^ \infty} \left\| \int_{t_k}^{t_{k+1}} \sigma( t_{k+1}+\cdot, s, \widetilde{X}_s) \rd W_s \right\|_{L^ \infty}^3 \right] \\
&\quad + C \hE \left[ \left\| \int_{t_k}^{t_{k+1}} \sigma( t_{k+1}+\cdot, s, \widetilde{X}_s) \rd W_s \right\|_{L^ \infty}^4 \right] \\
&\leq C(1+[\varphi]_{Lip}^4)h^2.
\end{align*}
In addition, \eqref{eq.2order3} can be obtained by a similar step as the proof of \eqref{eq.2order2}. Thus, the proof of this lemma is completed.
\end{proof}

\begin{theorem}
\label{thm.theta}
If Assumption \ref{assum.theta1} holds, then for any $f \in \cC_{b}^{4} (\mathbb{R}^d; \mathbb{R})$, the weak error of the stochastic theta method \eqref{eq:SthetaSVIE} can be controlled as 
\begin{align*}
\left| \mathbb{E}\, f(\bar{X}_T)-\mathbb{E}\, f(X_T) \right| 
\leq C h, 
\end{align*}
where $C>0$ is independent of $h$. 
\end{theorem}

\begin{proof}
According to the fundamental weak convergence theorem (i.e., Theorem \ref{Thm.ImportantII}), the proof is completed by Lemmas \ref{lem.theta2} and \ref{lem.theta1}. 
\end{proof}

\begin{remark} \label{remark.compare}
In the case of $\theta = 0$, the stochastic theta method \eqref{eq:SthetaSVIE} coincides, up to a minor modification, with the Euler-type scheme from \cite{BrasFukasawa2025}, making Theorem \ref{thm.theta} a generalization that recovers their main findings. To handle the one-step weak error, \cite{BrasFukasawa2025} relied on a generalized It\^o-type formula.
Due to the regularity conditions required by this formula, their framework necessitates a bounded diffusion coefficient $\sigma$ and $\cC^5$-continuity of the coefficients $\mu$ and $\sigma$. Notably, our approach circumvents this generalized It\^o formula, eliminating the boundedness assumption on $\sigma$ and relaxing the fifth-order differentiability requirement of $\mu$ and $\sigma$ on the state variable.
\end{remark}

\section{Wong--Zakai approximation}
\label{sec.WZA}
Wong--Zakai approximations are widely utilized to regularize noise and approximate stochastic systems. While recent works \cite{Kamrani2024, XuZhang2024} have investigated the strong error of such approximations for SVIEs, a corresponding weak error analysis is still lacking. In this section, we employ the fundamental weak convergence theorem (Theorem \ref{Thm.ImportantII}) to analyze the weak convergence rate of the Wong--Zakai approximation for the SVIE \eqref{eq.SVIE}.

We consider the linear interpolation
\begin{align*}
W_t^N = W_{t_j} + \frac{t - t_j}{h} \Delta W(t_j), \qquad t \in [t_j, t_{j+1}], \quad j \in \{0, 1, \cdots, N-1 \},
\end{align*}
where $\Delta W(t_j) = W_{t_{j+1}} - W_{t_j}$, and 
\begin{align} \label{eq.dWN}
\rd W_t^N = \frac{1}{h} \Delta W(t_j) \rd t.
\end{align} 
Then the Wong--Zakai approximation $\{\bar{\bar{X}}_t^N\}_{0 \leq t \leq T}$ for SVIE \eqref{eq.SVIE} is defined as
\begin{align} \label{eq.def.WZ}
\bar{\bar{X}}_t = X_0 + \int_0^t \mu(t, s, \bar{\bar{X}}_s) \rd s + \int_0^t \sigma(t, s, \bar{\bar{X}}_s) \rd W_s^N, \qquad t \in [0, T],
\end{align}
where the second integral is understood in the Lebesgue--Stieltjes sense.

Next, we show that the Wong--Zakai approximation \eqref{eq.def.WZ} can be recast within the framework developed in Subsection \ref{sub:FWCT}. For $j \in \{0,1,\cdots,N-1\}$, we apply the Wong--Zakai approximation to \eqref{eq.def:P_rt} with $r =h$ and $t = t_j$ and define the resulting operator $\vec{P}_{h,t_j}: \Phi_{2T} \to \Phi_{2T}$ by $\varphi\mapsto \vec{P}_{h,t_j}(\varphi)$ with 
\begin{small}
\begin{align} \label{eq.vecP_ht}
\vec{P}_{h, t_j}(\varphi)_u 
= \varphi_{h+u}+ \int_{t_j}^{t_{j+1}} \mu(t_{j+1}+u, s, \vec{X}_s^{t_j,\varphi} ) \rd s+ \int_{t_j}^{t_{j+1}} \sigma(t_{j+1}+u, s, \vec{X}_s^{t_j,\varphi} ) \rd W_s^N, \quad u \geq 0. 
\end{align}
\end{small}
Here, $\{ \vec X_{v}^{t_j,\varphi} \}_{v \in [t_j,t_{j+1}]}$ is the solution to 
\begin{align} 
\vec X_{v}^{t_j,\varphi} 
&= X_0 + \varphi_{v-t_j} + \int_{t_j}^{v} \mu(v, s, \vec X_{s}^{t_j,\varphi}) \rd s+ \int_{t_j}^{v} \sigma(v,s, \vec X_{s}^{t_j,\varphi}) \rd W_s^N. \label{eq.def:vecX}
\end{align} 
Let $k \in \{1,2,\cdots,N\}$. According to \eqref{eq.semigroup:tildeX} and \eqref{eq.vecP_ht}, one can obtain 
\begin{align*}
\vec{Y}_{t_k}(u) 
&= \vec{P}_{h,t_{k-1}}(\vec{Y}_{t_{k-1}})_{u} \\
&= \vec{Y}_{t_{k-1}}(h+u) + \int_{t_{k-1}}^{t_{k}} \mu(t_{k}+u, s, \vec X_{s}^{t_{k-1},\vec{Y}_{t_{k-1}}}) \rd s + \int_{t_{k-1}}^{t_{k}} \sigma(t_{k}+u, s, \vec X_{s}^{t_{k-1},\vec{Y}_{t_{k-1}}} ) \rd W_s^N \\
&= \vec{Y}_{t_{k-1}}(h+u) + \int_{t_{k-1}}^{t_{k}} \mu(t_{k}+u, s, \bar{\bar{X}}_{s}) \rd s + \int_{t_{k-1}}^{t_{k}} \sigma(t_{k}+u, s, \bar{\bar{X}}_{s} ) \rd W_s^N,
\end{align*}
where the fact that $\vec{X}_{s}^{t_{k-1},\vec{Y}_{t_{k-1}}} = \bar{\bar{X}}_{s}$ is also used in the last step. Similarly, 
\begin{align*}
\vec{Y}_{t_{k-1}}(h+u) = \vec{Y}_{t_{k-2}}(2h+u) + \int_{t_{k-2}}^{t_{k-1}} \mu(t_{k}+u, s, \bar{\bar{X}}_{s}) \rd s + \int_{t_{k-2}}^{t_{k-1}} \sigma(t_{k}+u, s, \bar{\bar{X}}_{s} ) \rd W_s^N.
\end{align*} 
Then, by iterating the above recurrence relation and using the fact that $\vec{Y}_0 = \widetilde{\mathbf{0}}$, one arrives at 
\begin{align} \label{eq.vecP_tk_0}
\vec{Y}_{t_k}(u) = \int_0^{t_k} \mu(t_k+u, s, \bar{\bar{X}}_s) \rd s + \int_0^{t_k} 
 \sigma(t_k+u, s, \bar{\bar{X}}_s ) \rd W_s^N, \qquad u \geq 0. 
\end{align}
Finally, the approximation derived from \eqref{eq:Xtk} coincides with the Wong--Zakai approximation \eqref{eq.def.WZ}.

\begin{assumption} \label{assum.WZ}
There is a constant $C := C(T) > 0$ such that for any $x \in \hR^d$ and any $t, s \in [0,T]$, 
\begin{align}
& \sigma(t, t, x)=0 \in \hR^{d \times m}, \label{eq:WZKasp1} \\
& |\partial_1 \mu (t,s,x)| \leq C(1+|x|), \quad |\partial_1 \sigma (t,s,x)| \leq C. \label{eq.WZasp3}
\end{align}
\end{assumption}

For example, Assumption \ref{assum.WZ} is satisfied in the scalar case when $\mu(t,s,x) = (t-s)x$ and $\sigma(t,s,x) = (t-s)\cos(x)$. Under Assumption \ref{assum.WZ}, the following proposition can be established using standard arguments in \cite{XuZhang2024} together with the H\"older and Gr\"onwall inequalities, and its proof is omitted.

\begin{proposition} 
Let $p \geq 2$ and $j \in \{0,1,\dots,N-1\}$. Let $X$ and $\bar{\bar{X}}$ denote the solutions of \eqref{eq.SVIE} and \eqref{eq.def.WZ}, respectively. For $t \in [0,T]$ and $\varphi \in \Phi_{2T} \cap \mathcal{C}^1([0,\infty); \mathbb{R}^d)$, denote by $\widetilde{X}^{t_j,\varphi}$ and $\vec{X}^{t_j,\varphi}$ the processes defined in \eqref{eq.def:tildeX} and \eqref{eq.def:vecX}, respectively. Then under Assumptions \ref{assum.theta1} and \ref{assum.WZ}, there exists a constant $C := C(T) > 0$, independent of $\varphi$ and $h$, such that for all $t \in [0,T]$ and $ t_j \leq v \leq t_{j+1}$,
\begin{align}
& \hE|X_t|^p + \hE|\bar{\bar{X}}_t|^p \leq C, \label{eq.WZprop1} \\
& \hE|\widetilde{X}_v^{t_j,\varphi}|^p + \hE|\vec{X}_v^{t_j,\varphi}|^p \leq C (1+ [\varphi]_{Lip}^p). \label{eq.WZprop2} 
\end{align}
\end{proposition}

According to Theorem \ref{Thm.ImportantII}, the first-order weak convergence of the Wong--Zakai approximation \eqref{eq.def.WZ} for SVIE \eqref{eq.SVIE} can be obtained by verifying Assumptions \ref{assum.barP is C^1} and \ref{assum.importantII} with $\nu = 4$ and $q = 1$, as shown in Lemmas \ref{lem.WZ2} and \ref{lem.WZ1}, respectively.

\begin{lemma} \label{lem.WZ2}
Let $k \in \{0,1,\cdots,N-1\}$. Under Assumptions \ref{assum.theta1} and \ref{assum.WZ}, there is a constant $C > 0$ independent of $k$ and $h$ such that $\vec P_{t_k, 0}(\widetilde{\mathbf{0}}) \in \cC^1 ([0, \infty); \hR^d)$ and $\mathbb{E} \left( \big[ \vec P_{t_k, 0} (\widetilde{\mathbf{0}}) \big]_{Lip}^4 \right) \leq C$. 
\end{lemma}

\begin{proof}
In view of the relation $\vec{P}_{t_k,0}(\widetilde{\mathbf{0}}) = \vec{Y}_{t_k}$ and \eqref{eq.vecP_tk_0}, we have that for any $u \geq 0$, 
\begin{align*}
\vec P_{t_k, 0}(\widetilde{\mathbf{0}})_u = \int_0^{t_k} \mu(t_k+u, s, \bar{\bar{X}}_s) \rd s + \int_0^{t_k} 
 \sigma(t_k+u, s, \bar{\bar{X}}_s ) \rd W_s^N.
\end{align*}
Then, it follows from Assumption \ref{assum.theta1} and Fubini's theorem that 
\begin{align*}
\vec P_{t_k, 0}(\widetilde{\mathbf{0}})_u
&= \int_0^{t_k} \int_0^u \partial_1 \mu(t_k+v, s, \bar{\bar{X}}_s) \rd v \rd s + \int_0^{t_k} \mu(t_k, s, \bar{\bar{X}}_s) \rd s \\
&\quad + \int_0^{t_k} \int_0^u \partial_1 \sigma(t_k+v, s, \bar{\bar{X}}_s) \rd v \rd W_s^N + \int_0^{t_k} \sigma(t_k, s, \bar{\bar{X}}_s) \rd W_s^N \\
&= \int_0^u \int_0^{t_k} \partial_1 \mu(t_k+v, s, \bar{\bar{X}}_s) \rd s \rd v + \int_0^{t_k} \mu(t_k, s, \bar{\bar{X}}_s) \rd s \\
&\quad + \int_0^u \int_0^{t_k} \partial_1 \sigma(t_k+v, s, \bar{\bar{X}}_s) \rd W_s^N \rd v + \int_0^{t_k} \sigma(t_k, s, \bar{\bar{X}}_s) \rd W_s^N, 
\end{align*}
which implies that $\vec P_{t_k, 0}(\widetilde{\mathbf{0}}) \in \cC^1 ([0, \infty); \hR^d)$ and 
\begin{align*}
\frac{\rd}{\rd u} \vec P_{t_k, 0}(\widetilde{\mathbf{0}})_u
= \int_0^{t_k} \partial_1 \mu(t_k+u, s, \bar{\bar{X}}_s) \rd s + \int_0^{t_k} \partial_1 \sigma(t_k+u, s, \bar{\bar{X}}_s) \rd W_s^N.
\end{align*}
Hence, using Assumption \ref{assum.WZ}, H\"older's inequality, \eqref{eq.dWN}, and \eqref{eq.WZprop1} yields 
\begin{align*}
\mathbb{E} \left( \big[ \vec P_{t_k, 0} (\widetilde{\mathbf{0}}) \big]_{Lip}^4 \right)
= \hE \bigg[ \Big( \sup_{u \geq 0} \Big| \frac{\rd}{\rd u}\vec P_{t_k, 0}(\widetilde{\mathbf{0}})_u \Big| \Big)^4 \bigg] 
\leq C.
\end{align*}
The proof is completed. 
\end{proof}

\begin{lemma} \label{lem.WZ1}
Let $g \in \cC_b^{4} ( \Phi_{2T}; \hR)$. If Assumptions \ref{assum.theta1} and \ref{assum.WZ} hold, then there exists a constant $C$ such that for any $k \in \{ 0, 1, \dots, N-1 \}$, $\varphi \in \Phi_{2T}\cap\cC^1 ([0, \infty); \hR^d)$, $m \in \{1, 2, 3\}$, and $\xi \in [0, 1]$, 
\begin{align}
&\left| \hE \left\< \nabla^m g( \tau_h \varphi ), (P_{h,t_k}(\varphi)- \tau_h \varphi )^{\otimes m} - (\vec P_{h,t_k}(\varphi)- \tau_h \varphi )^{\otimes m} \right\> \right| 
\leq C(1 + [\varphi]_{Lip}^4)h^{2}, \label{eq.WZ2order1} \\
&\left| \hE \left\<\nabla^{4} g( \tau_h \varphi + \xi (P_{h,t_k}(\varphi)- \tau_h \varphi )), (P_{h,t_k}(\varphi)- \tau_h \varphi )^{\otimes 4} \right\> \right| 
\leq C(1 + [\varphi]_{Lip}^4)h^{2}, \label{eq.WZ2order2} \\
&\left| \hE \left\<\nabla^{4} g( \tau_h \varphi + \xi (\vec P_{h,t_k}(\varphi)- \tau_h \varphi )), (\vec P_{h,t_k}(\varphi)- \tau_h \varphi )^{\otimes 4} \right\> \right| 
\leq C(1 + [\varphi]_{Lip}^4)h^{2}, \label{eq.WZ2order3} 
\end{align}
where $C > 0$ depends on $g$ only through its semi-norm $\|g\|_{\cC_b^4( \Phi_{2T})}$, but is independent of $\varphi$ and $h$. 
\end{lemma}

\begin{proof}
Throughout this proof, we denote $\widetilde{X}:=\widetilde{X}^{t_k,\varphi}$ and $\vec{X}:=\vec{X}^{t_k,\varphi}$ for short. It follows from \eqref{eq.def:vecX} that for $v \in [t_k, t_{k+1} )$ and sufficiently small $\varepsilon > 0$, 
\begin{align*}
\vec{X}_{v+\varepsilon}-\vec{X}_v 
&= \varphi_{v+\varepsilon -t_k}-\varphi_{v-t_k}+ \int_v^{v+\varepsilon} \mu(v+\varepsilon, s, \vec{X}_s) \rd s + \int_{t_k}^v \mu(v+\varepsilon, s, \vec{X}_s) - \mu(v, s, \vec{X}_s) \rd s \\ 
&\quad + \int_v^{v+\varepsilon} \sigma(v+\varepsilon, s, \vec{X}_s) \rd W_s^N + \int_{t_k}^v \sigma(v+\varepsilon, s, \vec{X}_s) - \sigma(v, s, \vec{X}_s) \rd W_s^N, 
\end{align*}
which implies the differential 
\begin{align*}
\rd \vec{X}_v 
&= \dot{\varphi}_{v-t_k} \rd v + \mu(v, v, \vec{X}_v) \rd v + \left( \int_{t_k}^v \partial_1 \mu(v, s, \vec{X}_s) \rd s \right) \rd v \\ 
&\quad + \sigma(v, v, \vec{X}_v) \rd W_v^N + \left( \int_{t_k}^v \partial_1 \sigma(v, s, \vec{X}_s) \rd W_s^N \right) \rd v. 
\end{align*} 
Then, by \eqref{eq:WZKasp1}, for $t_k \leq t' < t \leq t_{k+1}$, 
\begin{align} \label{eq.increment vecX}
\vec{X}_t - \vec{X}_{t'} 
&= \int_{t'}^t \dot{\varphi}_{v-t_k} \rd v + \int_{t'}^t \mu(v, v, \vec{X}_v) \rd v + \int_{t'}^t \int_{t_k}^v \partial_1 \mu(v, s, \vec{X}_s) \rd s \rd v \notag \\
&\quad + \int_{t'}^t \int_{t_k}^v \partial_1 \sigma(v, s, \vec{X}_s) \rd W_s^N \rd v. 
\end{align}
By H\"older's inequality, the Lipschitz continuity of $\varphi$, the linear growth property of $\mu$, \eqref{eq.WZasp3}, \eqref{eq.dWN}, and \eqref{eq.WZprop2}, one has
\begin{align} 
\hE |\vec{X}_t - \vec{X}_{t'}|^{4} 
&\leq Ch^3 \int_{t'}^t \hE (|\dot{\varphi}_{v-t_k}|^4) \rd v + Ch^3 \int_{t'}^t \hE(1+|\vec{X}_v|^4) \rd v \notag \\
&\quad + Ch^6 \int_{t'}^t \int_{t_k}^v \hE(1+ |\vec{X}_s|^4) \rd s \rd v + Ch^{6} \int_{t'}^t \int_{t_k}^v \hE \left( \left| \frac{\rd W_s^N}{\rd s} \right|^{4} \right) \rd s \rd v \notag \\
&\leq C[\varphi]_{Lip}^{4} h^{4} + C(1+[\varphi]_{Lip}^4)h^4 + C(1+[\varphi]_{Lip}^4)h^{8} + Ch^6 \notag\\
&\leq C (1+[\varphi]_{Lip}^4) h^{4}. \label{eq.regularity vecX}
\end{align}
Similarly, from \eqref{eq.dtildeX}, using H\"older's inequality, the Burkholder--Davis--Gundy inequality, \eqref{eq:WZKasp1}, \eqref{eq.WZasp3}, and \eqref{eq.WZprop2}, one obtains
\begin{align} \label{eq.regularity tildeX}
\hE |\widetilde{X}_t - \widetilde{X}_{t'}|^4 \leq C(1+[\varphi]_{Lip}^4)h^4, \qquad t_k \leq t' < t \leq t_{k+1}.
\end{align}
In addition, note that $\widetilde X_{t_k} = \vec X_{t_k}$. By \eqref{eq.regularity vecX} and \eqref{eq.regularity tildeX}, 
\begin{align} \label{eq.strong}
\hE |\widetilde X_s - \vec X_s |^{4} \leq 8 \hE |\widetilde X_s - \widetilde X_{t_k} |^{4} + 8 \hE |\vec X_{t_k} - \vec X_s |^{4} \leq C (1+[\varphi]_{Lip}^{4})h^{4}, \quad t_k < s \leq t_{k+1}.
\end{align}

\textit{Proof of \eqref{eq.WZ2order1} with $m = 1$}. Using \eqref{eq.def:P_rt} and \eqref{eq.vecP_ht} shows 
\begin{align*}
P_{h,t_k}(\varphi)- \tau_h \varphi 
&= \int_{t_k}^{t_{k+1}} \mu( t_{k+1}+\cdot, s, \widetilde{X}_s) \rd s + \widetilde{H}(\cdot), \\ 
\vec P_{h,t_k}(\varphi)- \tau_h \varphi
&= \int_{t_k}^{t_{k+1}} \mu(t_{k+1}+\cdot, s, \vec X_s ) \rd s + \vec H(\cdot), 
\end{align*} 
where $\widetilde{H}(\cdot)$ and $\vec H(\cdot)$ are defined by 
\begin{align} \label{eq.defH}
\widetilde{H}(\cdot) = \int_{t_k}^{t_{k+1}} \sigma (t_{k+1}+\cdot, s, \widetilde{X}_{s}) \rd W_{s}, 
\qquad
\vec H(\cdot) = \int_{t_k}^{t_{k+1}} \sigma (t_{k+1}+\cdot, s, \vec{X}_{s}) \rd W_{s}^N.
\end{align}

Using the triangle inequality and \eqref{eq.dWN} yields
\begin{align*}
&\quad\, \left| \hE \< \nabla g( \tau_h \varphi ), (P_{h,t_k}(\varphi)- \tau_h \varphi ) - (\vec P_{h,t_k}(\varphi)- \tau_h \varphi ) \> \right| \\
&\leq \left| \int_{t_k}^{t_{k+1}} \hE \< \nabla g( \tau_h \varphi ), \mu(t_{k+1}+\cdot, s, \widetilde X_{s}) - \mu(t_{k+1}+\cdot, s, \vec X_{s}) \> \rd s \right| \\
&\quad + \left| \int_{t_k}^{t_{k+1}} \hE \left\< \nabla g( \tau_h \varphi ), \frac{\Delta W(t_k)}{h} \sigma(t_{k+1}+\cdot,s,\vec{X}_s) \right\> \rd s \right| \\
&=: (\mathrm{M_I})_1 + (\mathrm{M_I})_2.
\end{align*}
For $(\mathrm{M_I})_1$, since $\nabla g$ and $\nabla_x \mu$ are bounded, together with \eqref{eq.strong}, one has $(\mathrm{M_I})_1 \leq C (1+[\varphi]_{Lip})h^2 $.

For $(\mathrm{M_I})_2$, one can arrive at
\begin{small}
\begin{align*}
&\quad\ (\mathrm{M_I})_2 \\
&\leq \left| \int_{t_k}^{t_{k+1}} \hE \big\< \nabla g( \tau_h \varphi ), \frac{\Delta W(t_k)}{h} \sigma(t_{k+1}+\cdot, s, \vec X_{t_k}) \big\> \rd s \right| \\
& + \left| \int_{t_k}^{t_{k+1}} \hE \big\< \nabla g( \tau_h \varphi ), \frac{\Delta W(t_k)}{h} \nabla_x \sigma(t_{k+1}+\cdot, s, \vec X_{t_k}) (\vec X_s - \vec X_{t_k}) \big\> \rd s \right| \\
& + \left| \int_{t_k}^{t_{k+1}} \hE \big\< \nabla g( \tau_h \varphi ), \frac{\Delta W(t_k)}{h} \int_0^1 \<\nabla_{xx} \sigma(t_{k+1}+\cdot, s, \vec X_{t_k} + \xi (\vec X_s - \vec X_{t_k}) ), (\vec X_s - \vec X_{t_k})^{\otimes 2} \> (1-\xi) \rd \xi \big\> \rd s \right| \\
&=: (\mathrm{M_I})_{2,1} + (\mathrm{M_I})_{2,2} + (\mathrm{M_I})_{2,3},
\end{align*}
\end{small}
where we used the Taylor expansion 
\begin{align*}
\sigma(t_{k+1}+\cdot, s, \vec X_s) 
&= \sigma(t_{k+1}+\cdot, s, \vec X_{t_k}) + \nabla_x \sigma(t_{k+1}+\cdot, s, \vec X_{t_k}) (\vec X_s - \vec X_{t_k}) \\
&\quad + \int_0^1 \<\nabla_{xx} \sigma(t_{k+1}+\cdot, s, \vec X_{t_k} + \xi (\vec X_s - \vec X_{t_k}) ), (\vec X_s - \vec X_{t_k})^{\otimes 2} \> (1-\xi) \rd \xi. 
\end{align*} 
For $(\mathrm{M_I})_{2,1}$, since $\Delta W(t_k)$ is independent of $\cF_{t_k}$, while $\sigma(t_{k+1}+u, s, \vec X_{t_k}) \in \cF_{t_k}$, one has
\begin{align*}
\hE [\Delta W(t_k) \sigma(t_{k+1}+u, s, \vec X_{t_k})] = \hE[\Delta W(t_k)] \hE[\sigma(t_{k+1}+u, s, \vec X_{t_k})] =0, \qquad u\geq0. 
\end{align*}
Hence, $(\mathrm{M_I})_{2,1}=0$.

For $(\mathrm{M_I})_{2,2}$, it follows from \eqref{eq.increment vecX} that 
\begin{small}
\begin{align*}
(\mathrm{M_I})_{2,2} 
&\leq \left| \int_{t_k}^{t_{k+1}} \hE \left\< \nabla g( \tau_h \varphi ), \frac{\Delta W(t_k)}{h} \nabla_x \sigma(t_{k+1}+\cdot, s, \vec X_{t_k}) \int_{t_k}^s \dot{\varphi}_{v-t_k} \rd v \right\> \rd s \right| \\
&\quad + \left| \int_{t_k}^{t_{k+1}} \hE \left\< \nabla g( \tau_h \varphi ), \frac{\Delta W(t_k)}{h} \nabla_x \sigma(t_{k+1}+\cdot, s, \vec X_{t_k}) \int_{t_k}^s \mu(v, v, \vec{X}_v) \rd v \right\> \rd s \right| \\
&\quad + \left| \int_{t_k}^{t_{k+1}} \hE \left\< \nabla g( \tau_h \varphi ), \frac{\Delta W(t_k)}{h} \nabla_x \sigma(t_{k+1}+\cdot, s, \vec X_{t_k}) \int_{t_k}^s \int_{t_k}^v \partial_1 \mu(v, r, \vec{X}_r) \rd r \rd v \right\> \rd s \right| \\
&\quad + \left| \int_{t_k}^{t_{k+1}} \hE \left\< \nabla g( \tau_h \varphi ), \frac{\Delta W(t_k)}{h} \nabla_x \sigma(t_{k+1}+\cdot, s, \vec X_{t_k}) \int_{t_k}^s \int_{t_k}^v \partial_1 \sigma(v, r, \vec{X}_r) \rd W_r^N \rd v \right\> \rd s \right| \\
&=:(\mathrm{M_I})_{2,2,1} + (\mathrm{M_I})_{2,2,2} + (\mathrm{M_I})_{2,2,3} + (\mathrm{M_I})_{2,2,4}.
\end{align*}
\end{small}
Analogous to $(\mathrm{M_I})_{2,1}$, it holds that $(\mathrm{M_I})_{2,2,1}=0$. For $(\mathrm{M_I})_{2,2,2}$, one can deduce that
\begin{small}
\begin{align*}
(\mathrm{M_I})_{2,2,2}
&\leq \left| \int_{t_k}^{t_{k+1}} \hE \left\< \nabla g( \tau_h \varphi ), \frac{\Delta W(t_k)}{h} \nabla_x \sigma(t_{k+1}+\cdot, s, \vec X_{t_k}) \int_{t_k}^s \mu(v, v, \vec{X}_{t_k}) \rd v \right\> \rd s \right| \\
& + \left| \int_{t_k}^{t_{k+1}} \hE \left\< \nabla g( \tau_h \varphi ), \frac{\Delta W(t_k)}{h} \nabla_x \sigma(t_{k+1}+\cdot, s, \vec X_{t_k}) \int_{t_k}^s \big( \mu(v, v, \vec{X}_v) - \mu(v, v, \vec{X}_{t_k}) \big) \rd v \right\> \rd s \right| \\
&= \left| \int_{t_k}^{t_{k+1}} \hE \left\< \nabla g( \tau_h \varphi ), \frac{\Delta W(t_k)}{h} \nabla_x \sigma(t_{k+1}+\cdot, s, \vec X_{t_k}) \int_{t_k}^s \big( \mu(v, v, \vec{X}_v) - \mu(v, v, \vec{X}_{t_k}) \big) \rd v \right\> \rd s \right| \\
&\leq C (1+[\varphi]_{Lip}) h^\frac{5}{2}, 
\end{align*}
\end{small}
\!\!where the boundedness of $\nabla g$ and $\nabla_x \sigma$, the Cauchy--Schwarz inequality, and \eqref{eq.regularity vecX} are used in last step. For $(\mathrm{M_I})_{2,2,3}$, since $\nabla g$ and $\nabla_x \sigma$ are bounded, combining the Cauchy--Schwarz inequality, \eqref{eq.WZasp3}, and \eqref{eq.WZprop2}, one obtains 
\begin{align*}
(\mathrm{M_I})_{2,2,3} 
&\leq \frac{C}{h} \int_{t_k}^{t_{k+1}} \int_{t_k}^s \int_{t_k}^v \hE[ |\Delta W(t_k)| (1+|\vec{X}_r|) ] \rd r \rd v \rd s \\
&\leq \frac{C}{h} \int_{t_k}^{t_{k+1}} \int_{t_k}^s \int_{t_k}^v \big(\hE |\Delta W(t_k)|^2\big)^\frac{1}{2} \big(\hE (1+|\vec{X}_r|)^2\big)^\frac{1}{2} \rd r \rd v \rd s \\
&\leq C (1+[\varphi]_{Lip}) h^\frac{5}{2}.
\end{align*}
For $(\mathrm{M_I})_{2,2,4}$, applying the boundedness of $\nabla g$, $\nabla_x \sigma$, and $\partial_1 \sigma$, together with \eqref{eq.dWN}, indicates 
\begin{align*}
(\mathrm{M_I})_{2,2,4} 
&\leq \frac{C}{h} \int_{t_k}^{t_{k+1}} \int_{t_k}^s \int_{t_k}^v \hE \left[ |\Delta W(t_k)| \left| \frac{\rd W_r^N}{\rd r} \right| \right] \rd r \rd v \rd s \\
&= \frac{C}{h^2} \int_{t_k}^{t_{k+1}} \int_{t_k}^s \int_{t_k}^v \hE |\Delta W(t_k)|^2 \rd r \rd v \rd s \\
&\leq Ch^2.
\end{align*}
Thus, one can arrive at $(\mathrm{M_I})_{2,2} \leq C (1+[\varphi]_{Lip}) h^2$.

For $(\mathrm{M_I})_{2,3}$, since $\nabla g$ and $\nabla_{xx} \sigma$ are bounded, by the Cauchy--Schwarz inequality and \eqref{eq.regularity vecX}, one has
\begin{align*}
(\mathrm{M_I})_{2,3} 
&\leq \frac{C}{h} \int_{t_k}^{t_{k+1}} \hE [ |\Delta W(t_k)| |\vec X_s - \vec X_{t_k}|^2] \rd s \\
&\leq \frac{C}{h} \int_{t_k}^{t_{k+1}} (\hE |\Delta W(t_k)|^2)^\frac{1}{2} \left(\hE |\vec X_s - \vec X_{t_k}|^4 \right)^\frac{1}{2} \rd s \\
&\leq C (1+[\varphi]_{Lip}^2) h^\frac{5}{2}.
\end{align*}
Thus, one can conclude that $(\mathrm{M_I})_{2} \leq C (1+[\varphi]_{Lip}^2) h^2$, which, together with $(\mathrm{M_I})_1 \leq C (1+[\varphi]_{Lip})h^2$, implies 
\begin{align*}
\left| \hE \< \nabla g( \tau_h \varphi ), (P_{h,t_k}(\varphi)- \tau_h \varphi ) - (\vec P_{h,t_k}(\varphi)- \tau_h \varphi ) \> \right| \leq C (1+[\varphi]_{Lip}^2) h^2.
\end{align*}

\textit{Proof of \eqref{eq.WZ2order1} with $m = 2$}. 
Using \eqref{eq.def:P_rt}, \eqref{eq.vecP_ht}, and the boundedness of $\nabla^2 g$ shows 
\begin{align} \label{eq.WZm=2}
\left | \hE \left\< \nabla^2 g( \tau_h \varphi ), (P_{h,t_k}(\varphi)- \tau_h \varphi )^{\otimes 2} - (\bar P_{h,t_k}(\varphi)- \tau_h \varphi )^{\otimes 2} \right\> \right | 
\leq C \sum_{i=1}^4 (\mathrm{M_{II}})_{i} 
\end{align}
with
\begin{align*}
& (\mathrm{M_{II}})_{1} := \hE \bigg[ \left\| \int_{t_k}^{t_{k+1}} \mu( t_{k+1}+\cdot, s, \widetilde{X}_s) \rd s \right\|_{L^ \infty}^2 + \left\| \int_{t_k}^{t_{k+1}} \mu(t_{k+1}+\cdot, s, \vec X_s ) \rd s \right\|_{L^ \infty}^2 \bigg], \\
& (\mathrm{M_{II}})_{2} := \hE \left[ \left\| \int_{t_k}^{t_{k+1}} \mu( t_{k+1}+\cdot, s, \widetilde{X}_s) - \mu( t_{k+1}+\cdot, s, \vec{X}_{t_k}) \rd s \right\|_{L^ \infty} \big\| \widetilde{H} \big\|_{L^ \infty} \right], \\
& (\mathrm{M_{II}})_{3} := \hE \left[ \left\| \int_{t_k}^{t_{k+1}} \mu( t_{k+1}+\cdot, s, \vec{X}_{t_k}) \rd s \right\|_{L^ \infty} \big\| \widetilde{H} - \vec H \big\|_{L^ \infty} \right], \\
& (\mathrm{M_{II}})_{4} := \left| \hE \left\< \nabla^2 g( \tau_h \varphi ), \widetilde{H}^{\otimes 2} - \vec H^{\otimes 2} \right\> \right|, 
\end{align*} 
where $\widetilde{H}$ and $\vec{H}$ are defined in \eqref{eq.defH}. For $(\mathrm{M_{II}})_{1}$, it follows from \eqref{eq.Appendix lem2_1} and \eqref{eq.WZprop2} that 
\begin{align} \label{eq.WZm=2_1}
(\mathrm{M_{II}})_{1} \leq
C(1+[\varphi]_{Lip}^2)h^2.
\end{align}
For $(\mathrm{M_{II}})_{2}$, using the Cauchy--Schwarz inequality, together with \eqref{eq.Appendix lem2_2}, \eqref{eq.WZprop2}, and \eqref{eq.strong} yields 
\begin{align} \label{eq.WZm=2_2} 
(\mathrm{M_{II}})_{2} \leq C (1+[\varphi]_{Lip}^2) h^\frac{5}{2}. 
\end{align}

To facilitate the analysis of $(\mathrm{M_{II}})_{3}$, we deal with 
\begin{align*}
\hE \big[ \| \widetilde{H}-\vec{H} \|_{L^ \infty}^2 \big]
&= \hE \left[ \left\| \int_{t_k}^{t_{k+1}} \big[ \sigma( t_{k+1}+\cdot, s, \widetilde{X}_s) \rd W_s - \sigma( t_{k+1}+\cdot, s, \vec{X}_s) \rd W_s^N \big] \right\|_{L^ \infty}^2 \right] \\
& \leq C \hE \left[ \left\| \int_{t_k}^{t_{k+1}} \sigma( t_{k+1}+\cdot, s, \widetilde{X}_s) -\sigma( t_{k+1}+\cdot, t_k, \widetilde{X}_s) \rd W_s \right\|_{L^ \infty}^2 \right] \\
&\quad + C \hE \left[ \left\| \int_{t_k}^{t_{k+1}} \sigma( t_{k+1}+\cdot, t_k, \widetilde{X}_s) - \sigma( t_{k+1}+\cdot, t_k, \vec{X}_{t_k}) \rd W_s \right\|_{L^ \infty}^2 \right] \\
&\quad + C \hE \left[ \left\| \int_{t_k}^{t_{k+1}} \sigma( t_{k+1}+\cdot, t_k, \vec{X}_{t_k}) (\rd W_s - \rd W_s^N) \right\|_{L^ \infty}^2 \right] \\
&\quad + C \hE \left[ \left\| \int_{t_k}^{t_{k+1}} \sigma( t_{k+1}+\cdot, t_k, \vec{X}_{t_k}) - \sigma( t_{k+1}+\cdot, s, \vec{X}_{t_k}) \rd W_s^N \right\|_{L^ \infty}^2 \right] \\
&\quad + C \hE \left[ \left\| \int_{t_k}^{t_{k+1}} \sigma( t_{k+1}+\cdot, s, \vec{X}_{t_k}) - \sigma( t_{k+1}+\cdot, s, \vec{X}_s) \rd W_s^N \right\|_{L^ \infty}^2 \right] \\
&=:(\mathrm{V})_1 + (\mathrm{V})_2 + (\mathrm{V})_3 + (\mathrm{V})_4 + (\mathrm{V})_5.
\end{align*}
Applying Lemma \ref{Appendix lem3}, \eqref{eq.WZprop2}, and \eqref{eq.strong} reveals that $(\mathrm{V})_1 + (\mathrm{V})_2 \leq C(1+[\varphi]_{Lip}^4) h^3$. From \eqref{eq.dWN}, one obtains $\int_{t_k}^{t_{k+1}} \rd W_s = \int_{t_k}^{t_{k+1}} \rd W_s^N $, which immediately yields $(\mathrm{V})_3=0$. Using \eqref{eq.dWN}, \eqref{eq:asp1}, the Cauchy--Schwarz inequality, and \eqref{eq.WZprop2} shows 
\begin{align*}
(\mathrm{V})_4 &\leq Ch \int_{t_k}^{t_{k+1}} \hE \left[ |t_k-s|^2 (1+|\vec{X}_{t_k}|)^2 \left| \frac{\rd W_s^N}{\rd s} \right|^2 \right] \rd s \\
&\leq Ch^3 \int_{t_k}^{t_{k+1}} \hE \left[ (1+|\vec{X}_{t_k}|)^2 \left| \frac{\rd W_s^N}{\rd s} \right|^2 \right] \rd s \\
&\leq Ch \int_{t_k}^{t_{k+1}} (\hE (1+|\vec{X}_{t_k}|)^4)^\frac{1}{2} ( \hE | \Delta W(t_k) |^4 )^\frac{1}{2} \rd s \\
&\leq C (1+[\varphi]_{Lip}^2) h^3.
\end{align*}
By a similar step as in the above estimate, together with \eqref{eq.regularity vecX}, one obtains $(\mathrm{V})_5 \leq C (1+[\varphi]_{Lip}^2) h^3$. Thus,
\begin{align} \label{eq.dW_s-dW_s^N}
\hE \big[ \| \widetilde{H}-\vec{H} \|_{L^ \infty}^2 \big] \leq C(1+[\varphi]_{Lip}^4) h^3, 
\end{align} 
which together with the Cauchy--Schwarz inequality, \eqref{eq.Appendix lem2_1}, and \eqref{eq.WZprop2} yields
\begin{align} \label{eq.WZm=2_3} 
(\mathrm{M_{II}})_{3} \leq C (1+[\varphi]_{Lip}^3) h^\frac{5}{2}. 
\end{align}

To estimate $(\mathrm{M_{II}})_{4}$, we note that by \eqref{eq.Appendix lem2_2}, \eqref{eq.Appendix lem2_3}, and \eqref{eq.WZprop2}, 
\begin{align} \label{eq.H_moment}
\hE \left[ \big\| \widetilde{H} \big\|_{L^\infty}^p + \big\| \vec H \big\|_{L^\infty}^p \right] 
\leq C (1+[\varphi]_{Lip}^p) h^{\frac{p}{2}}, \qquad p \geq 2.
\end{align}
Then applying the boundedness of $\nabla^2 g$, the Cauchy--Schwarz inequality, \eqref{eq.H_moment}, and \eqref{eq.dW_s-dW_s^N} shows 
\begin{align}
(\mathrm{M_{II}})_{4} 
\leq C \hE \left[ \left( \big\| \widetilde{H} \big\|_{L^\infty} + \big\| \vec H \big\|_{L^\infty} \right) \big\| \widetilde{H}-\vec{H} \big\|_{L^ \infty} \right] 
\leq C (1+[\varphi]_{Lip}^3) h^2. \label{eq.WZm=2_4}
\end{align}
By substituting the estimates \eqref{eq.WZm=2_1}, \eqref{eq.WZm=2_2}, \eqref{eq.WZm=2_3}, and \eqref{eq.WZm=2_4} into \eqref{eq.WZm=2}, we complete the proof of \eqref{eq.WZ2order1} in the case $m=2$.

\textit{Proof of \eqref{eq.WZ2order1} with $m = 3$}. Using \eqref{eq.def:P_rt} and \eqref{eq.vecP_ht}, together with the assumption that $\nabla^3 g$ is bounded, one has 
\begin{align} \label{eq.WZm=3}
\left | \hE \left\< \nabla^3 g( \tau_h \varphi ), (P_{h,t_k}(\varphi) - \tau_h \varphi )^{\otimes 3} - (\vec P_{h,t_k}(\varphi)- \tau_h \varphi )^{\otimes 3} \right\> \right | 
\leq C \sum_{i=1}^4 (\mathrm{M_{III}})_{i} 
\end{align}
with
\begin{small}
\begin{align*}
&(\mathrm{M_{III}})_{1} := \hE \left\| \int_{t_k}^{t_{k+1}} \mu( t_{k+1}+\cdot, s, \widetilde{X}_s) \rd s \right\|_{L^ \infty}^3 + \hE \left\| \int_{t_k}^{t_{k+1}} \mu( t_{k+1}+\cdot, s, \vec{X}_s) \rd s \right\|_{L^ \infty}^3, \\
&(\mathrm{M_{III}})_{2} := \hE \bigg[ \left\| \int_{t_k}^{t_{k+1}} \mu( t_{k+1}+\cdot, s, \widetilde{X}_s) \rd s \right\|_{L^ \infty}^2 \big\| \widetilde{H} \big\|_{L^ \infty} + \left\| \int_{t_k}^{t_{k+1}} \mu(t_{k+1}+\cdot, s, \vec X_s ) \rd s \right\|_{L^ \infty}^2 \big\| \vec H \big\|_{L^ \infty} \bigg], \\
&(\mathrm{M_{III}})_{3} := \hE \bigg[ \left\| \int_{t_k}^{t_{k+1}} \mu( t_{k+1}+\cdot, s, \widetilde{X}_s) \rd s \right\|_{L^ \infty} \big\| \widetilde{H} \big\|_{L^ \infty}^2 + \left\| \int_{t_k}^{t_{k+1}} \mu(t_{k+1}+\cdot, s, \vec X_s ) \rd s \right\|_{L^ \infty} \big\| \vec H \big\|_{L^ \infty}^2 \bigg], \\
&(\mathrm{M_{III}})_{4} := \left| \hE \left\< \nabla^3 g( \tau_h \varphi ), \widetilde{H}^{\otimes 3} - \vec H^{\otimes 3} \right\> \right|, 
\end{align*}
\end{small}
where $\widetilde{H}$ and $\vec{H}$ are defined in \eqref{eq.defH}. It follows from the Cauchy--Schwarz inequality, Lemma \ref{Appendix lem2}, and \eqref{eq.WZprop2} that
\begin{align} \label{eq.WZm=3_1}
(\mathrm{M_{III}})_{1} + (\mathrm{M_{III}})_{2} + (\mathrm{M_{III}})_{3} \leq C(1+[\varphi]_{Lip}^3)h^2.
\end{align}

For $(\mathrm{M_{III}})_{4}$, with the definition \eqref{eq.defH} at hand, using the boundedness of $\nabla^3 g$ and the Cauchy--Schwarz inequality yields 
\begin{align*}
(\mathrm{M_{III}})_{4} 
&\leq | \hE \< \nabla^3 g(\varphi), \widetilde{H}^{\otimes 2} \otimes ( \widetilde{H}- \vec{H} ) \> | + | \hE \< \nabla^3 g(\varphi), \widetilde{H} \otimes ( \widetilde{H}- \vec{H} ) \otimes \vec{H} \> | \\
&\quad + | \hE \< \nabla^3 g(\varphi), ( \widetilde{H}- \vec{H} ) \otimes \vec{H}^{\otimes 2} \> | \\
&\leq C \hE \big[ \| \widetilde{H} \|_{L^ \infty}^2 \| \widetilde{H}-\vec{H} \|_{L^ \infty} \big] + C \hE \big[\| \widetilde{H} \|_{L^ \infty} \| \widetilde{H}-\vec{H} \|_{L^ \infty} \| \vec{H} \|_{L^ \infty} \big] \\
&\quad + C \hE \big[ \| \widetilde{H}-\vec{H} \|_{L^ \infty} \| \vec{H} \|_{L^ \infty}^2 \big] \\
&\leq C ( \hE \| \widetilde{H} \|_{L^ \infty}^4 )^{1/2} ( \hE \| \widetilde{H}-\vec{H} \|_{L^ \infty}^2 )^{1/2} \\
&\quad + C ( \hE \| \widetilde{H} \|_{L^ \infty}^4 )^{1/4} ( \hE \| \widetilde{H}-\vec{H} \|_{L^ \infty}^2 )^{1/2} ( \hE \| \vec{H} \|_{L^ \infty}^4 )^{1/4} \\
&\quad + C ( \hE \| \widetilde{H}-\vec{H} \|_{L^ \infty}^2 )^{1/2} ( \hE \| \vec{H} \|_{L^ \infty}^4 )^{1/2}.
\end{align*}
Thus, by \eqref{eq.dW_s-dW_s^N} and \eqref{eq.H_moment}, 
\begin{align} \label{eq.WZm=3_2}
(\mathrm{M_{III}})_{4} \leq C(1+[\varphi]_{Lip}^4)h^\frac{5}{2}.
\end{align} 
Combining the estimates \eqref{eq.WZm=3}, \eqref{eq.WZm=3_1}, and \eqref{eq.WZm=3_2}, we complete the proof of \eqref{eq.WZ2order1} in the case $m=3$.

\textit{Proof of \eqref{eq.WZ2order3}}. By the boundedness of $\nabla^4 g$ and \eqref{eq.vecP_ht}, together with the Cauchy--Schwarz inequality, \eqref{eq.Appendix lem2_1}, \eqref{eq.WZprop2}, and \eqref{eq.H_moment}, one obtains
\begin{align*}
&\quad\, \left| \hE \left\<\nabla^4 g( \tau_h \varphi + \eta (\vec P_{h,t_k}(\varphi)- \tau_h \varphi )), (\vec P_{h,t_k}(\varphi)- \tau_h \varphi )^{\otimes 4} \right\> \right| \\
& \leq C \hE \left[ \left\| \int_{t_k}^{t_{k+1}} \mu( t_{k+1}+\cdot, s, \vec{X}_s) \rd s \right\|_{L^ \infty}^4 \right] + C \hE \left[ \left\| \int_{t_k}^{t_{k+1}} \mu( t_{k+1}+\cdot, s, \vec{X}_s) \rd s \right\|_{L^ \infty}^3 \big\| \vec H \big\|_{L^ \infty} \right] \\
&\quad + C \hE \left[ \left\| \int_{t_k}^{t_{k+1}} \mu( t_{k+1}+\cdot, s, \vec{X}_s) \rd s \right\|_{L^ \infty}^2 \big\| \vec H \big\|_{L^ \infty}^2 \right] \\
&\quad + C \hE \left[ \left\| \int_{t_k}^{t_{k+1}} \mu( t_{k+1}+\cdot, s, \vec{X}_s) \rd s \right\|_{L^ \infty} \big\| \vec H \big\|_{L^ \infty}^3 \right] + C \hE \left[ \big\| \vec H \big\|_{L^ \infty}^4 \right] \\
&\leq C(1+[\varphi]_{Lip}^4)h^2.
\end{align*}
In addition, \eqref{eq.WZ2order2} follows from the proof of \eqref{eq.2order2}. Therefore, the proof of this lemma is complete.
\end{proof}

\begin{theorem} 
If Assumptions \ref{assum.theta1} and \ref{assum.WZ} hold, then for any $f \in \cC_{b}^{4} (\mathbb{R}^d; \mathbb{R})$, the weak error of the Wong--Zakai approximation \eqref{eq.def.WZ} can be controlled as 
\begin{align*}
\left| \mathbb{E}\, f({\bar{\bar{X}}}_T)-\mathbb{E}\, f(X_T) \right| 
\leq C h, 
\end{align*}
where $C>0$ is independent of $h$. 
\end{theorem}

\begin{proof}
According to the proposed fundamental weak convergence theorem (i.e., Theorem \ref{Thm.ImportantII}), the proof is completed by Lemmas \ref{lem.WZ2} and \ref{lem.WZ1}. 
\end{proof}

\begin{remark}
Notice that the condition \eqref{eq:WZKasp1} in Assumption \ref{assum.WZ} is also required in the strong error analysis of the Wong--Zakai approximation \eqref{eq.def.WZ}; see \cite{XuZhang2024}.
\end{remark}

\section{Numerical experiments}
\label{sec.NumericalExperiments}

In this section, we apply the stochastic theta method \eqref{eq:SthetaSVIE} to a generalized Stein--Stein stochastic volatility model introduced in \cite{BrasFukasawa2025} and numerically investigate its weak convergence rate. The numerical results will confirm the theoretical findings of Theorem \ref{thm.theta}.

Consider the generalized Stein--Stein stochastic volatility model 
\begin{align*}
\begin{cases}
S_t = S_0 + \int_0^t S_s V_s \rd B_s, \\
V_t = V_0 + g_0(t) + \kappa \int_0^t K(t-s) V_s \rd s + \nu \int_0^t K(t-s) \sigma(V_s) \rd W_s,
\end{cases}
\end{align*}
where $t\in[0,T]$, and $K(t)=A_1(A_2+t)^{-1/4}$ is a shifted power-law kernel. The driving Brownian motion $B_t$ is given by $B_t = \rho W_t + \sqrt{1-\rho^2}\,W^\perp_t$, where $\rho\in[-1,1]$ is the correlation coefficient and $W^\perp$ is a standard Brownian motion independent of $W$.

\begin{table}[htbp]
\footnotesize
\caption{The weak error $\left| \hE\, f (\bar{S}_T) - \hE\, f (S_T) \right|$ and convergence order when $\sigma \equiv 1$.}
\label{tab.additive noise} 
\begin{center}
\setlength{\tabcolsep}{9mm}
\begin{tabular}{cccc} 
\hline
Stepsize & $\theta = 0$ & $\theta = 0.5$ & $\theta = 1$ \\
\hline
\ \ $2^{-5}$ & $2.4299e$-$04$ & $2.5509e$-$04$ & $2.6724e$-$04$ \\
\ \ $2^{-6}$ & $1.0843e$-$04$ & $1.1480e$-$04$ & $1.2118e$-$04$ \\
\ \ $2^{-7}$ & $5.4607e$-$05$ & $5.7850e$-$05$ & $6.1104e$-$05$ \\
\ \ $2^{-8}$ & $2.9971e$-$05$ & $3.1575e$-$05$ & $3.3183e$-$05$ \\
\ \ $2^{-9}$ & $1.5365e$-$05$ & $1.6119e$-$05$ & $1.6875e$-$05$ \\
\hline
Weak order & $0.9821$ & $0.9831$ & $0.9839$ \\
\hline
\end{tabular}
\end{center}
\end{table}

\begin{table}[htbp]
\footnotesize
\caption{The weak error $\left| \hE\, f (\bar{S}_T) - \hE\, f (S_T) \right|$ and convergence order when $\sigma(x) = x$.}
\label{tab.multiplicative noise}
\begin{center}
\setlength{\tabcolsep}{9mm}
\begin{tabular}{cccc} 
\hline
Stepsize & $\theta = 0$ & $\theta = 0.5$ & $\theta = 1$ \\
\hline
\ \ $2^{-5}$ & $2.5009e$-$04$ & $2.6223e$-$04$ & $2.7441e$-$04$ \\
\ \ $2^{-6}$ & $1.1544e$-$04$ & $1.2184e$-$04$ & $1.2828e$-$04$ \\
\ \ $2^{-7}$ & $6.1261e$-$05$ & $6.4539e$-$05$ & $6.7836e$-$05$ \\
\ \ $2^{-8}$ & $3.0934e$-$05$ & $3.2558e$-$05$ & $3.4201e$-$05$ \\
\ \ $2^{-9}$ & $1.3979e$-$05$ & $1.4738e$-$05$ & $1.5511e$-$05$ \\
\hline
Weak order & $1.0222$ & $1.0210$ & $1.0197$ \\
\hline
\end{tabular}
\end{center}
\end{table}

In all experiments, the test function is chosen as $f(x)=(x-\mathcal{K})_+$, which represents the payoff function of a European call option with strike price $\mathcal{K}\ge0$. To test the method comprehensively, we consider both additive noise (where $\sigma(V_s) = 1$) and multiplicative noise (where $\sigma(V_s) = V_s$). The shared model parameters are chosen as follows:
\begin{align*}
& T=1,\quad S_0=1,\quad V_0=0.1,\quad \rho=-0.7,\quad \kappa=0.01,\quad \nu=0.05, \\
& A_1=0.3,\quad A_2=0.02,\quad \mathcal{K}=1,\quad g_0(t) = \frac{0.04}{3 A_1} t^{3/4}.
\end{align*}

The simulation results are collected in Tables \ref{tab.additive noise} and \ref{tab.multiplicative noise}. As observed from the empirical data, the weak convergence rate of the stochastic theta method \eqref{eq:SthetaSVIE} is approximately 1, thereby validating the theoretical claims of Theorem \ref{thm.theta}.

\appendix

\section{Auxiliary lemmas}
\label{appendix_AuxiLemm}

\begin{lemma} \label{lem.boundedness of moments}
Let $p \geq 2$ and $0 \leq r_1 < r_2 \leq T$. For $u \geq 0$, define the It\^o-type stochastic integral 
\begin{align*}
\Psi_u := \int_{r_1}^{r_2} F(u, s) \rd W_s, 
\end{align*}
where the integrand $\{F(u, s)\}_{u \geq 0, \, s \in [0, T]}$ is an $\mathbb{R}^{d \times m}$-valued stochastic process adapted to the filtration $\hF$ with respect to its second variable. Assume that $F(u, s)=0$ for $u > 2T$, and that with $C_F > 0$, 
\begin{align*}
\left\| F(u_1, s) - F(u_2, s) \right\|_{L^p(\Omega; \mathbb{R}^{d \times m})} \leq C_F |u_1 - u_2|, \qquad \forall\, u_1,\, u_2 \geq 0,\quad s \in[0, T]. 
\end{align*}
Then there exists a continuous modification $\tilde{\Psi}$ of $\Psi$ satisfying 
\begin{align*}
\hE \big[ \| \tilde{\Psi} \|_{L^ \infty}^p \big]
:= \hE \big[ \sup\nolimits_{u \geq 0} | \tilde{\Psi}_u |^p \big] 
\leq C C_F^p (r_2 - r_1)^{p/2}, 
\end{align*}
where $C := C (p,T) > 0$ is independent of $C_F$ and $(r_2 - r_1)$.
\end{lemma}

\begin{proof}
Applying the Burkholder--Davis--Gundy inequality shows that for $u_1, u_2 \ge 0$, 
\begin{align*}
\| \Psi_{u_1} - \Psi_{u_2} \|_{L^p(\Omega; \mathbb{R}^{d})}^p
&= \left\| \int_{r_1}^{r_2} F(u_1, s) - F(u_2, s) \rd W_s \right\|_{L^p(\Omega; \mathbb{R}^{d})}^p \\
&\leq C \left( \int_{r_1}^{r_2} \left\| F(u_1, s) - F(u_2, s) \right\|_{L^p(\Omega; \mathbb{R}^{d \times m})}^2 \rd s \right)^\frac{p}{2} \\
&\leq C C_F^p |u_1 - u_2|^p (r_2-r_1)^\frac{p}{2}.
\end{align*}
Then using the Kolmogorov continuity theorem (cf.\ \cite[Theorem A.1]{BrasFukasawa2025}) indicates that there exists a modification $\tilde{\Psi}$ of $\Psi$ which is $\alpha$-H\"older continuous for every $\alpha \in (0, (p-1)/p)$ and 
\begin{align*}
\hE \left[ \left( \sup_{\substack{u_1 \neq u_2\\ u_1,u_2\in[0,3T]}} \frac{|\tilde{\Psi}_{u_1} -\tilde{\Psi}_{u_2}|}{|u_1-u_2|^\alpha} \right)^p \right] 
\leq C_T C_F^p (r_2-r_1)^\frac{p}{2}.
\end{align*}
Finally, taking $u_1=3T$ and $u_2=u$ with $\Psi_{3T}=0$ yields
\begin{align*}
\hE \big[ \| \tilde{\Psi} \|_{L^ \infty}^p \big]
= \hE \big[ \sup\nolimits_{u \geq 0} | \tilde{\Psi}_u |^p \big] 
= \hE \big[ \sup\nolimits_{u \in [0,2T]} | \tilde{\Psi}_u |^p \big]
\leq C_TC_F^p (r_2-r_1)^{p/2}, 
\end{align*}
which completes the proof. 
\end{proof}

\begin{lemma} 
\label{Appendix lem2}
Let $p \geq 2$ and $Z \in \cC([0,T]; L^p(\Omega; \hR^d))$ be an $\hF$-adapted process. If the linear growth condition
\begin{align} \label{appen:linearGrowth}
|\mu (t, s, x)| + |\sigma (t, s, x)| 
\leq C (1+|x|), \qquad \forall\, t,s \in [0,T],~~x \in \mathbb{R^d} 
\end{align}
and \eqref{eq:asp2} hold, then for any $k \in \{0,1,\cdots,N-1\}$, 
\begin{align}
&\hE \left[ \left\| \int_{t_k}^{t_{k+1}} \mu( t_{k+1}+\cdot, s, Z_s) \rd s \right\|_{L^ \infty}^p \right]
\leq C \left( 1 + \sup_{t \in [0,T]} \|Z_t\|_{L^p(\Omega; \hR^d)}^p \right) h^p, 
\label{eq.Appendix lem2_1} \\
&\hE \left[ \left\| \int_{t_k}^{t_{k+1}} \sigma( t_{k+1}+\cdot, s, Z_s) \rd W_s \right\|_{L^ \infty}^p \right] 
\leq C \left( 1 + \sup_{t \in [0,T]} \|Z_t\|_{L^p(\Omega; \hR^d)}^p \right) h^\frac{p}{2}, 
\label{eq.Appendix lem2_2} \\
&\hE \left[ \left\| \int_{t_k}^{t_{k+1}} \sigma( t_{k+1}+\cdot, s, Z_s) \rd W_s^N \right\|_{L^ \infty}^p \right] 
\leq C \left( 1 + \sup_{t \in [0,T]} \|Z_t\|_{L^{2p}(\Omega; \hR^d)}^p \right) h^\frac{p}{2}. \label{eq.Appendix lem2_3} 
\end{align}
\end{lemma}

\begin{proof}
Applying H\"older's inequality together with \eqref{appen:linearGrowth} yields 
\begin{align*}
\hE \left[ \left\| \int_{t_k}^{t_{k+1}} \mu( t_{k+1}+\cdot, s, Z_s) \rd s \right\|_{L^ \infty}^p \right]
&\leq C h^{p-1} \int_{t_k}^{t_{k+1}} \hE \left[ \left\| \mu( t_{k+1}+\cdot, s, Z_s) \right\|_{L^ \infty}^p \right] \rd s \\
&= Ch^{p-1} \int_{t_k}^{t_{k+1}} \hE \left[ \left(\sup_{u \geq 0} \left| \mu( t_{k+1}+u, s, Z_s) \right| \right)^p \right] \rd s\\
&\leq C \left( 1 + \sup_{t \in [0,T]} \|Z_t\|_{L^p(\Omega; \hR^d)}^p \right) h^p, 
\end{align*}
which shows that \eqref{eq.Appendix lem2_1} holds.

We define $F(u,s) = \sigma(t_{k+1}+u, s, Z_s)$ for $u \geq 0$ and $s \in [t_k, t_{k+1}]$. Using \eqref{eq:asp2} implies that 
\begin{align*}
\left\| F(u_1, s) - F(u_2, s) \right\|_{L^p(\Omega; \hR^{d \times m})} \leq C \left( 1 + \sup_{t \in [0,T]} \|Z_t\|_{L^p(\Omega; \hR^d)} \right) |u_1 - u_2|, 
\end{align*}
which together with Lemma \ref{lem.boundedness of moments} proves \eqref{eq.Appendix lem2_2}.

By H\"older's inequality and \eqref{appen:linearGrowth}, one can arrive at 
\begin{align*}
&\quad\ \hE \left[ \left\| \int_{t_k}^{t_{k+1}} \sigma( t_{k+1}+\cdot, s, Z_s) \rd W_s^N \right\|_{L^ \infty}^p \right] \\
&= \hE \left[ \left\| \int_{t_k}^{t_{k+1}} \sigma( t_{k+1}+\cdot, s, Z_s) \frac{\rd W_s^N}{\rd s} \rd s \right\|_{L^ \infty}^p \right] \\
&\leq Ch^{p-1} \int_{t_k}^{t_{k+1}} \hE \left[ \left\| \sigma( t_{k+1}+\cdot, s, Z_s) \frac{\rd W_s^N}{\rd s} \right\|_{L^ \infty}^p \right] \rd s \\
&\leq C h^{p-1} \int_{t_k}^{t_{k+1}} \left( \hE \left[ 1 + \sup_{t \in [0,T]}|Z_t|^{2p} \right] \right)^\frac{1}{2} \left( \hE \left| \frac{\Delta W(t_k)}{h} \right|^{2p} \right)^\frac{1}{2} \rd s \\
&\leq C \left( 1 + \sup_{t \in [0,T]} \|Z_t\|_{L^{2p}(\Omega; \hR^d)}^p \right) h^\frac{p}{2},
\end{align*}
which shows that \eqref{eq.Appendix lem2_3} holds. The proof is completed. 
\end{proof}

\begin{lemma}
\label{Appendix lem3}
Let $p \geq 2$, $\sigma \in \cC^{0,1,1}([0, T] \times [0, T] \times \mathbb{R}^d; \hR^{d \times m})$ and $Z, \bar{Z} \in \cC([0,T]; L^p(\Omega; \hR^d))$ be two $\hF$-adapted processes. Under condition \eqref{eq:asp3}, it holds that for any $k \in \{0,1,\cdots,N-1\}$, 
\begin{align}
&\hE \left[ \left\| \int_{t_k}^{t_{k+1}} \sigma( t_{k+1}+\cdot, s, Z_s) - \sigma( t_{k+1}+\cdot, t_k, Z_s) \rd W_s \right\|_{L^ \infty}^p \right] \label{eq.Appendix lem3_1} \\
&\qquad \leq C \left( 1 + \sup_{t \in [0,T]} \|Z_t\|_{L^p(\Omega; \hR^d)}^p \right) h^{3p/2}, \notag \\
&\hE \left[ \left\| \int_{t_k}^{t_{k+1}} \sigma( t_{k+1}+\cdot, t_k, Z_s) - \sigma( t_{k+1}+\cdot, t_k, \bar{Z}_{t_k}) \rd W_s \right\|_{L^ \infty}^p \right] \label{eq.Appendix lem3_2} \\
&\qquad \leq C \sup_{s \in [t_k,t_{k+1}]} \|Z_s-\bar{Z}_{t_k}\|_{L^{2p}(\Omega; \hR^d)}^p \left( 1+ \sup_{t \in [0,T]} \|Z_t\|_{L^{2p}(\Omega; \hR^d)}^p + \sup_{t \in [0,T]} \|\bar{Z}_t\|_{L^{2p}(\Omega; \hR^d)}^p \right) h^{p/2}. \notag
\end{align}
\end{lemma}

\begin{proof}
We denote $F(u,s) = \sigma( t_{k+1}+u, s, Z_s) - \sigma( t_{k+1}+u, t_k, Z_s)$ for $u \geq 0$ and $s \in [t_k, t_{k+1}]$. By the mean value theorem and \eqref{eq:asp3}, one obtains
\begin{small}
\begin{align*}
&\quad\ \| F(u_1,s) - F(u_2,s) \|_{L^p(\Omega; \hR^{d \times m})} \\
&= \big\| \sigma( t_{k+1}+u_1, s, Z_s) - \sigma( t_{k+1}+u_1, t_k, Z_s) - \big( \sigma( t_{k+1}+u_2, s, Z_s) - \sigma( t_{k+1}+u_2, t_k, Z_s) \big) \big\|_{L^p(\Omega; \hR^{d \times m})} \\
&= \left\| (s-t_k) \int_0^1 \partial_2 \sigma(t_{k+1}+u_1, t_k+\xi(s-t_k), Z_s) - \partial_2 \sigma(t_{k+1}+u_2, t_k+\xi(s-t_k), Z_s) \rd \xi \right\|_{L^p(\Omega; \hR^{d \times m})} \\
&\leq Ch \left(1+ \sup_{t \in [0,T]} \|Z_t\|_{L^p(\Omega; \hR^d)} \right) |u_1-u_2|.
\end{align*}
\end{small}
Hence, \eqref{eq.Appendix lem3_1} follows from Lemma \ref{lem.boundedness of moments}.

Similarly, letting $F(u,s)=\sigma( t_{k+1}+u, t_k, Z_s) - \sigma( t_{k+1}+u, t_k, \bar{Z}_{t_k})$ with $u \geq 0$ and $s \in [t_k, t_{k+1}]$ yields 
\begin{align*}
&\quad\ \| F(u_1,s) - F(u_2,s) \|_{L^p(\Omega; \hR^{d \times m})} \\
&= \big\| \big( \sigma( t_{k+1}+u_1, t_k, Z_s) - \sigma( t_{k+1}+u_1, t_k, \bar{Z}_{t_k}) \big ) \\
&\qquad\quad - \big( \sigma( t_{k+1}+u_2, t_k, Z_s) - \sigma( t_{k+1}+u_2, t_k, \bar{Z}_{t_k}) \big) \big\|_{L^p(\Omega; \hR^{d \times m})} \\
&= \bigg \| \int_0^1 \< \nabla_x \sigma(t_{k+1}+u_1, t_k, \bar{Z}_{t_k} + \xi(Z_s-\bar{Z}_{t_k})) - \nabla_x \sigma(t_{k+1}+u_2, t_k, \bar{Z}_{t_k} + \xi(Z_s-\bar{Z}_{t_k})), \\
&\qquad\quad\ Z_s-\bar{Z}_{t_k} \> \rd \xi \bigg\|_{L^p(\Omega; \hR^{d \times m})} \\
&\leq C \sup_{s \in [t_k,t_{k+1}]} \|Z_s-\bar{Z}_{t_k}\|_{L^{2p}(\Omega; \hR^d)} \left(1+ \sup_{t \in [0,T]} \|Z_t\|_{L^{2p}(\Omega; \hR^d)} + \sup_{t \in [0,T]} \|\bar{Z}_t\|_{L^{2p}(\Omega; \hR^d)} \right) |u_1-u_2|.
\end{align*} 
Thus, \eqref{eq.Appendix lem3_2} follows from Lemma \ref{lem.boundedness of moments}. The proof is completed. 
\end{proof}



\end{document}